\documentclass[12pt,reqno]{amsart}
\usepackage[margin=1in]{geometry}
\usepackage{hyperref}

\newcommand{\Ss}{{\mathbb S}}

\newcommand{\ord}{\operatorname{ord}}

\usepackage{mathrsfs, amsmath,amssymb,amsfonts, amsthm, dsfont, tikz-cd}
\newtheorem{theorem}{Theorem}[section]

\newtheorem{proposition}[theorem]{Proposition}
\newtheorem{corollary}[theorem]{Corollary}
\newtheorem{lemma}[theorem]{Lemma}
\newtheorem{remark}[theorem]{Remark}

\numberwithin{equation}{section}

\newcommand{\Z}{{\mathbb Z}}

\newcommand\WF{\operatorname{WF}'}

\numberwithin{equation}{section}

\newcommand{\R}{{\mathbb R}}

\usepackage{epsfig}
\usepackage{amscd}
\usepackage{tikz-cd}
\usepackage{amsmath, amssymb, amsthm}
\usepackage{graphics}
\usepackage{color}
\usepackage{dsfont}
\newtheorem*{main-theorem}{Main Theorem}
\newtheorem*{old-thm}{Theorem}
\theoremstyle{definition}
\numberwithin{equation}{section}
\newcommand\vol{\operatorname{vol}}

\def\11{\mathds{1}}

\def\WF{\mathrm{WF}\,}

\def\supp{\mathrm{supp}\,}

\def\phi{\varphi}

\def\be{\begin{eqnarray*}}
\def\ee{\end{eqnarray*}}
\def\ben{\begin{eqnarray}}
\def\een{\end{eqnarray}}

\def\L2R{L_{\text{Rest}}^2}

\newcommand{\ccal}{\mathcal{C}}

\newcommand{\pcal}{\mathcal{P}}

\newcommand{\codim}{\operatorname{codim}}
\newcommand{\coker}{\operatorname{coker}}
\newcommand{\im}{\operatorname{im}}
\newcommand{\T}{{\mathbb T}}

\begin{document}
\title{Fourier coefficients of restrictions of eigenfunctions}
\author{Emmett L. Wyman}
\address{Department of Mathematics, University of Rochester, Rochester NY, USA}
\email{emmett.wyman@rochester.edu}

\author{Yakun Xi}
\address{School of Mathematical Sciences, Zhejiang University, Hangzhou 310027, PR China}
\email{yakunxi@zju.edu.cn}

\author{Steve Zelditch}
\address{Department of Mathematics, Northwestern University, Chicago IL, USA}
\email{s-zelditch@northwestern.edu}

\maketitle

\begin{abstract} Let $\{e_j\}$ be an orthonormal basis of Laplace eigenfunctions of a compact Riemannian manifold $(M,g)$. Let $H \subset M$ be a submanifold and let $\{\psi_k\}$ be an orthonormal basis of Laplace eigenfunctions of $H$ with the induced metric.
We obtain joint asymptotics for the Fourier coefficients 
\[
	\langle \gamma_H e_j, \psi_k \rangle_{L^2(H)}  = \int_H e_j \overline \psi_k \, dV_H,
\]
of restrictions $\gamma_H e_j$ of $e_j$ to $H$. In particular, we obtain asymptotics for the sums of the  norm-squares of the Fourier coefficients over the joint spectrum  $\{(\mu_k, \lambda_j)\}_{j,k - 0}^{\infty}$ of the (square roots of the)  Laplacian $\Delta_M$ on $M$ and the
Laplacian $\Delta_H$ on $H$  in a family of suitably `thick' regions in $\R^2$. Thick regions include (1) the truncated cone $\mu_k/\lambda_j \in [a,b] \subset (0,1)$ and $\lambda_j \leq \lambda$, and (2) the slowly thickening strip $|\mu_k - c\lambda_j| \leq w(\lambda)$ and $\lambda_j \leq \lambda$, where $w(\lambda)$ is monotonic and $1 \ll w(\lambda) \lesssim \lambda^{1/2}$. Key tools for obtaining these asymptotics include the composition calculus of Fourier integral operators and a new multidimensional Tauberian theorem.

\end{abstract}

\tableofcontents

\section{Introduction}

\subsection{Background}

Let $(M,g) $ be a compact Riemannian manifold without boundary of dimension $n$ and let $\Delta_M$ denote the Laplace--Beltrami operator with respect to $g$. Let  $\{e_j\}_{j=0}^{\infty} $ be an orthonormal basis of eigenfunctions of $-\Delta_M$ with eigenvalues $ \lambda_j^2$ enumerated in increasing order, 
\[
	\Delta_M e_j = -\lambda_j^2 e_j, \qquad \langle e_j, e_k \rangle = \delta_{jk},
\]
where the inner product is $\langle f_1, f_2 \rangle_{L^2(M)}  = \int_M f_1 \overline{f}_2 dV_g$ with $dV_g$ the volume form of $g$. Let $H \subset M$ be an embedded, closed submanifold with dimension $d$, Riemannian metric $g_H$, and Laplacian $\Delta_H$. Let $\{\psi_k\}_{k=0}^{\infty}$
be an orthonormal basis of eigenfunctions of $\Delta_H$, with
\[
    \Delta_H \psi_k = -\mu_k^2 \psi_k.
\]
Denote by
\begin{equation} \label{FEXP}
	\gamma_H e_j = e_j |_H = \sum_k \langle \gamma_H e_j, \psi_k \rangle_{L^2(H)}  \psi_k \end{equation}
the expansion of the restriction of $e_j$ to $H$ in the basis $\psi_k$. We refer to the inner products  
\begin{equation} \label{FCDEF} 
    \int e_j \overline{\psi_k} \, dV_H,
\end{equation}
as the \emph{Fourier coefficients} of $\gamma_H e_j$; here,  $dV_H$ denotes the volume density on $H$. The purpose of this
article is to study the joint asymptotics of the Fourier coefficients \eqref{FCDEF} when the joint spectrum $\{(\lambda_j, \mu_k)\}_{j,k=0}^{\infty}$ falls into a family of suitably `thick' regions.

Ideally, one would like to have sharp bounds on the individual Fourier coefficients \eqref{FCDEF} of a subsequence $\{e_{j_k}\}_{k=1}^{\infty}$  of eigenfunctions with $\frac{\mu_k}{\lambda_j} = c$, as 
  $(M, g, H, c)$ vary over all
compact Riemannian manifolds, submanifolds and eigenvalue ratios. However, except in special cases (Section \ref{EXAMPLES}), it is difficult to extract asymptotic information about a single Fourier coefficient for  a subsequence $\{e_{j_k}\}_{k=1}^{\infty}$ of eigenfunctions.
Since
\[
	\sum_k \left| \int_H e_j \overline \psi_k \, dV_H \right|^2 = \int_H |e_j|^2 \, dV_H,
\]
 giving precise information on individual  Fourier coefficients is a substantial refinement on giving precisely information on the 
$L^2(H)$-restriction of $e_j$, which is itself difficult to describe for individual eigenfunctions (see \cite{BGT} and \cite{Hu} for such estimates).

In this article, we study  the asymptotics of Fourier coefficients  in an average sense. We consider the weak-$*$ limit of the measures
\begin{equation} \label{empirical measure}
	\nu = \lim_{\lambda \to \infty} \lambda^{-n} \sum_{j,k: \lambda_j \leq \lambda} \left| \int e_j \overline \psi_k \, dV_H \right|^2 \delta_{\mu_k/\lambda_j}
\end{equation}
in the dual of the bounded continuous functions on $\R$. We compute the limiting measure by obtaining asymptotics for the sum
\begin{equation}\label{sum}
	\sum_{\substack{\lambda_j \leq \lambda \\ \mu_k/\lambda_j \in [a,b]}} \left| \int e_j \overline \psi_k \, dV_H \right|^2
\end{equation}
in the style of the third author's result in \cite{Zel92}, where $[a,b]$ is a subinterval of $(0,1)$. We also obtain asymptotics for the ladder sum
\begin{equation}\label{ladder sum}
	\sum_{\substack{j,k: \lambda_j \leq \lambda \\ |\mu_k - c\lambda_j| \leq w(\lambda)}} \left| \int e_j \overline \psi_k \, dV_H \right|^2
\end{equation}
where the slope $c$ is fixed and lies in the interval $(0,1)$, and the width $w(\lambda)$ of the strip is some monotonic function which slowly grows to infinity with $\lambda$. To obtain both results, we introduce a basic multidimensional Tauberian theorem---Theorem \ref{main tauberian}---which the first and second authors hope to refine in later work.

The motivation to study Fourier expansions of restricted eigenfunctions    originated in the setting of automorphic eigenfunctions on 
hyperbolic surfaces, in particular the restriction of modular forms to closed geodesics, distance circles, and closed horocycles (in the finite area cusped case). Look to \cite{Ra40, Sel65, Br78, Br81, K80} for estimates and to \cite{Br81, G83, I97, I02} for a more systematic treatment of the topic. The Fourier coefficients in the negatively curved case  are expected to be rather uniform in the interval
$[0, 1]$. This statement is unproved and does not seem to have been formulated precisely before.   By comparison, the Fourier coefficients of the standard basis $Y^m_N$ of spherical harmonics of degree $N$ on the sphere $\Ss^2$ are  highly non-uniform. On a latitude circle,  
only the $m$ Fourier coefficient of $Y_N^m$  is non-zero; its  size depends on the relation between the ratio $\frac{m}{N}$ and the
latitude  (see Section \ref{2Sphere}). These observations  motivate the question of how 
the  dynamics of the geodesic flows of $(M,g)$ and of $(H, g |_H)$ determine the equidistribution properties
of the restricted Fourier coefficients \eqref{FCDEF}. When $\psi_k$ is fixed and only  $e_j$ vary, the Fourier coefficients are sometimes called `periods' of $e_j$ and were first studied under the name of Kuznecov sum formulae \cite{K80} in the general Riemannian context in \cite{Zel92}. While improvements to the remainder bounds of that article have yet to appear, the last 10 years have seen numerous improvements to bounds on the corresponding spectral projection operators (see e.g. \cite{SXZh17, WX, Xi, Xib, CG18, CG18b, CGT17}.) 
Fourier coefficients \eqref{FCDEF}, or `periods' in which $\psi_k$ varies, are sometimes called `generalized periods' (see \cite{Xi} for
 some results in the case of closed hypersurfaces.)

In the sequel \cite{WXZ20}, we obtain asymptotics for refined ladder sums \eqref{ladder sum} when $|\mu_k - c  \lambda_k| \leq w$ for constant width $w$.  It turns out
that the extremals for the individual term occur only when $c=1$.  The case where $c=1$ and $H$ is a totally geodesic submanifold
is studied in   \cite{WXZ20}.  In further work we also plan to study the case $c=1$ and $H$ has non-degenerate second fundamental form, which involves Airy type caustic effects.

\subsection{Statement of Results}

In what follows, $M$ and $H$ will be compact, boundary-less Riemannian manifolds of dimensions $n$ and $d$, respectively, with isometric embedding $H \to M$. We let $\Delta_M$ and $\Delta_H$ be the respective Laplace--Beltrami operators and $e_j$ for $j = 1,2,\ldots$ resp. $\psi_k$ for $k = 1,2,\ldots$ the corresponding orthonormal eigenbases with
\[
	\Delta_M e_j = -\lambda_j^2 e_j \qquad \text{ and } \qquad \Delta_H \psi_k = -\mu_k^2 \psi_k.
\]
Our main theorem is the following asymptotics for the sum \eqref{sum}.

\begin{theorem}\label{main cone} Let $[a,b] \subset (0,1)$. Then,
	\[
		\sum_{\substack{\lambda_j \leq \lambda \\ \mu_k/\lambda_j \in [a,b] }} \left| \int_H e_j \overline\psi_k \, dV_H \right|^2 = \frac{C_{H,M}}{n} \left(\int_{a}^{b} t^{d-1}(1 - t^2)^{\frac{n-d-2}{2}} \, dt \right) \lambda^n + O_{[a,b]}(\lambda^{n-1}),
	\]
	where we have constant
	\[
		C_{H,M} = (2\pi)^{-n}(\vol S^{d-1})(\vol S^{n-d-1})(\vol H).
	\]
	Moreover, the remainder is uniform for $[a,b]$ contained in a compact subset of $(0,1)$.
\end{theorem}

We acknowledge a minor abuse of notation in the theorem above, namely that the $e_j$ appearing in the integral denotes its pullback from $M$ to $H$ via the embedding $H \to M$. We will repeat this use throughout the article.

As a corollary to this theorem, we obtain a description of the empirical measure $\nu$ in \eqref{empirical measure}.

\begin{corollary} \label{main corollary}
	For any bounded, continuous function $f$ on $\R$, we have
	\[
		\lim_{\lambda \to \infty} \lambda^{-n} \sum_{j,k: \lambda_j \leq \lambda} \left| \int e_j \overline \psi_k \, dV_H \right|^2 f(\mu_k/\lambda_j) = \frac{C_{H,M}}{n} \int_0^1 f(t) t^{d-1} (1 - t^2)^{\frac{n-d-2}{2}} \, dt
	\]
	where the constant $C_{H,M}$ is the same as in Theorem \ref{main cone}. In other words, the limit in \eqref{empirical measure} indeed converges in the weak-$*$ limit in the dual of the bounded continuous functions, and the limit is
	\[
		d\nu(t) = \frac{C_{H,M}}{n} t^{d-1} (1 - t^2)^{\frac{n-d-2}{2}} \chi_{[0,1]}(t) \, dt
	\]
	where $\chi_{[0,1]}$ is the characteristic function of the interval $[0,1]$.
\end{corollary}

We also obtain asymptotics for slowly thickening ladder sums as in \eqref{ladder sum}.

\begin{theorem}\label{main ladder}
	Fix $c \in (0,1)$ and let $w(\lambda)$ be a monotone increasing function of $\lambda$ for which $w(\lambda) \to \infty$ and $w(\lambda) = O(\lambda^{1/2})$. Then,
	\[
		\sum_{\substack{j,k : \lambda_j \leq \lambda \\ |\mu_k - c\lambda_j| \leq w(\lambda)}} \left| \int e_j \overline \psi_k \, dV_H \right|^2 = \frac{2C_{H,M}}{n - 1} w(\lambda) c^{d-1} (1 - c^2)^{\frac{n-d-2}{2}} \lambda^{n-1} + O_{c,w}(\lambda^{n-1}).
	\]
\end{theorem}

Both Theorems \ref{main cone} and \ref{main ladder} follow as corollaries from the next two theorems. The first is an estimate on a smoothed version of the sum \eqref{sum}, which we prove in Section \ref{FIO} using FIO theory. The second is our basic multidimensional Tauberian theorem, which we prove in Section \ref{TAUBERIAN}.

\begin{theorem} \label{main FIO}
Consider the measure
\[
	N = \sum_{j,k} \left| \int_H e_j \overline \psi_k \, dV_H \right|^2 \delta_{(\mu_k, \lambda_j)}
\]
on $\R^2$. Let $\rho$ be a Schwartz-class function on $\R^2$. Fix $[a,b] \subset (0,1)$. For $(\mu,\lambda)$ in the cone $\mu/\lambda \in [a,b]$ with $\mu,\lambda > 0$, and for $\widehat \rho$ supported in a compact set depending on $H$, $M$, and $[a,b]$, we have asymptotics
\[
	N * \rho(\mu,\lambda) = C_{H,M} \widehat \rho(0) \mu^{d-1} \lambda^{n-d-1} (1 - \mu^2/\lambda^2)^{\frac{n-d-2}{2}} + O_{\rho,[a,b]}(\lambda^{n-3}).
\]
\end{theorem}

The main term in the asymptotics has a nice geometric interpretation. On the bundle $T^*_HM$ of covectors in $M$ over points in $H$, there is a natural volume density
\[
	\omega \in |T^*M| \otimes |M|^{-1} \otimes |H|,
\]
where here we are using the density notation of Duistermaat and Guillemin \cite{DG75}. Let $p_H$ and $p_M$ denote the principal symbols of $\sqrt{-\Delta_H}$ and $\sqrt{-\Delta_M}$, respectively, and let $i : T_H^*M \to T^*M$ be the inclusion and $\pi : T^*M \to T^*H$ the fiberwise projection. Then, given $\lambda$ and $\mu$, we have the Leray density on $\{p_M \circ i = \lambda\} \cap \{p_H \circ \pi = \mu\}$ given by
\[
	\frac{\omega}{|d(p_M \circ i) \wedge d(p_H \circ \pi)|}.
\]
For the sake of illustration, consider the model situation where $H = \R^d$, $M = \R^n$, where $H \to M$ is the embedding into $\R^d \times 0^{n-d}$, and where we take as symbols $p_H(x,\xi) = |\xi|$ and $p_M(z,\zeta) = |\zeta|$. Then one works out explicitly that $\omega = dx \, d\zeta$ and
\[
	\frac{\omega}{|d(p_M \circ i) \wedge d(p_H \circ \pi)|} = \frac{1}{\sqrt{1 - \mu^2/\lambda^2}} dx \, d\sigma
\]
where $\sigma$ is the restriction of the Euclidean volume to
\[
	\{ \zeta \in \R^n : |\zeta| = \lambda \text{ and } |\pi(\zeta)| = \mu \} = (\mu S^{d-1}) \times (\sqrt{\lambda^2 - \mu^2} S^{n-d-1}).
\]
Pretending momentarily that $\vol H$ is finite, integrating this volume element yields
\[
	(\vol H)(\vol S^{d-1})(\vol S^{n-d-1}) \mu^{d-1}\lambda^{n-d-1} (1 - \mu^2/\lambda^2)^{\frac{n-d-2}{2}}.
\]
This accounts for everything but $\widehat \rho(0)$ and the dimensional power of $2\pi$ in the main term of Theorem \ref{main FIO}. Indeed, this volume element will arise in the computation of the principal symbol of the Lagrangian distribution $\widehat N$ for the big singularity at the origin.

The Fourier Tauberian theorems are tools for obtaining asymptotics for a monotonic increasing function if we are given some information about (1) its derivative, and if we want finer remainders, (2) the order of the singularities of its Fourier transform away from the origin. For background on the one-dimensional Fourier Tauberian theorems, we refer the reader to Levitin's appendix of \cite{Lev97} and \cite{Saf} and the references therein. To state our Tauberian theorem, we borrow the definition of an \emph{order function} from \cite{Zwo}. Specifically, we say $m$ is an order function on $\R^d$ if it is positive and there exist a positive constants $C$ and power $\nu$ for which
\[
	m(x) \leq C(1 + |x - y|)^\nu m(y) \qquad \text{ for all } x,y \in \R^d.
\]

\begin{theorem}[Basic Multidimensional Tauberian Theorem] \label{main tauberian}
Let $N$ be a tempered, positive Radon measure on $\R^d$ and let $\rho$ be a nonnegative Schwartz-class function on $\R^d$ satisfying
\[
	\int_{\R^n} \rho(x) \, dx=1,
\]
and suppose
\[
	N * \rho(x) \leq m(x) \qquad \text{ for all } x \in \R^d
\]
for an order function $m$. Then for any Borel subset $\Omega$ of $\R^d$ with $N(\Omega) < \infty$, we have
\[
	\left|N(\Omega) - \int_\Omega N * \rho(x) \, dx \right| \leq C \int_{\partial^{[-1,1]}\Omega} m(x) \, dx
\]
where here $\partial^{[-1,1]}\Omega$ denotes the unit thickening of the boundary of $\Omega$ and the constant $C$ does not depend on $\Omega$.
\end{theorem}

Presently, there does not seem to be a systematic treatment of Fourier Tauberian theorems with functions of more than one parameter in the literature. The closest result the authors could find is due to Colin de Verdiere in \cite{CdV79}, where he obtains something resembling
\[
	N(\lambda \Omega) = \int_{\lambda \Omega} N * \rho(x) \, dx + O(\lambda^{\nu})
\]
where $\lambda \Omega$ denotes a scaling of $\Omega$ by $\lambda$ about the origin, and where $\Omega$ is compact and has piecewise $C^1$-boundary among some hypotheses on $N$. Note, while this result is set in $\R^d$, the family of regions is necessarily homothetic family indexed by a single real parameter. The main insight of Theorem \ref{main tauberian} is the connection between the remainder and the size of the boundary of $\Omega$. This allows us to obtain asymptotics for the ladders in Theorem \ref{main ladder}.

The paper is organized as follows. In section \ref{EXAMPLES}, we directly verify Theorem \ref{main cone} in the case where $H$ is a coordinate plane in the flat torus. The proofs of Corollary \ref{main corollary}, Theorem \ref{main FIO}, and Theorem \ref{main tauberian} are contained in sections \ref{COROLLARY}, \ref{FIO}, and \ref{TAUBERIAN}, respectively. The proofs of the corollary and the Tauberian theorem use only elementary tools. The proof in section \ref{FIO} relies on the symbol calculus of cleanly composing FIOs as presented in \cite{DG75} (see also \cite{HIV} for a thorough treatment and \cite{D96} for background).

\subsection*{Acknowledgements} Xi was partially supported by National Key R\&D Program of China No. 2022YFA1007200, National Natural Science Foundation of China No. 12171424. Wyman was partially supported by National Science Foundation of USA No. DMS-2204397 and by the AMS Simons travel grants. Zelditch was supported by National Science Foundation of USA Nos. DMS-1810747 and DMS-1502632. The authors are grateful to Madelyne Brown for pointing out an error in in an earlier draft of this paper. The authors are also grateful to the anonymous referees for their thorough and invaluable feedback. 


\section{Examples} \label{EXAMPLES}

In this section, we illustrate the definitions and results with two types of examples: (i) The standard $\Ss^2$, and (ii) flat tori.  

In particular,
we illustrate the nature of the parameter $c = \frac{\mu_k}{\lambda_j}$. 
The sum in \eqref{sum} is over the joint spectral points $(\mu_k, \lambda_j)$ lying in the set $\frac{\mu_k}{\lambda_j} \in [a,b]$. 
On the classical level, where we replace the eignvalues of the  operators by their principal symbols, this set corresponds to the 
 `wedges' or `cone',  
\begin{equation} \label{CWEDGE}
\ccal_{[a,b]} := \{(x, \xi) \in T^* M\setminus 0:  \frac{|\pi_{H_{\phi_0}}(x, \xi)| }{ |\xi|} = c \in [a,b] \}.
\end{equation}  
Below, we relate \eqref{CWEDGE} to wedges (or cones) around rays in the image of the moment map in these two examples. But
for general $(M,g, H)$ without symmetry, the wedge \eqref{CWEDGE} does not have such an interpretation. 

\subsection{Restrictions to curves in  $\Ss^2$}\label{2Sphere} Let $\Ss^2$ be the standard sphere. We illustrate the definitions
in the case where $H$ is a latitude circle (an orbit of the rotational action around the third axis) and for the standard basis $Y^{m}_{\ell}$
of spherical harmonics of degree $\ell$.

  Let $\frac{\partial}{\partial \theta} $ 
generate rotations around the $x_3$ axis in $\R^3$, and let  $(\theta, \phi)$ be the standard spherical coordinates. A latitude circle is a level
set of the azimuthal coordinate $H_{\phi_0} : \{\phi = \phi_0\}$ and the equator is the special case $\phi_0 = \pi/2$.   Since rotations commute with the geodesic flow,  the
Clairaut integral,
$$p_{\theta}(x, \xi) = \langle \xi, \frac{\partial}{\partial \theta} \rangle = |\frac{\partial}{\partial \theta} |_{H_{\phi_0} } \cos \angle \frac{\partial}{\partial \theta},
\dot{\gamma}_{x, \xi}(0), \;\; (x, \xi) \in T_x^*\Ss^2), $$
is a constant of the motion, i.e. the components of the moment map $\pcal: = (|\xi|, p_{\theta}): T^*\Ss^2 \to \R^2$ Poisson commute. 
Let $$u_{\theta}(\theta, \phi)  : =   \left| \frac{\partial}{\partial \theta} \right|^{-1}_{H_{\phi}} \;\frac{\partial}{\partial \theta} $$
be the unit vector field in the direction of $\frac{\partial}{\partial \theta}$ and let $\frac{\partial}{\partial \phi}$ be the unit vector field tangent to the
meridians.  Let $u_{\theta}^*, u_{\phi}^*$ be the dual unit coframe field.
The orthogonal projection from $T_{H_{\phi_0}} \Ss^2 \to T^* H_{\phi_0} $ is given by, $$ \pi_{H_{\phi_0}}(x, \xi)  = \langle \xi, u_{\theta} \rangle u^*_{\theta}.$$
{Thus, if we fix $H_{\phi_0}$ and $c \in [0,1]$, the corresponding slice of the cone $\ccal_{[a,b]}$ is given by: \begin{equation} \label{Tc} T^c_{H_{\phi_0}}: = \Big\{(x, \xi) \in T^*_{H_{\phi_0}} \Ss^2\setminus 0: \frac{|\pi_{H_{\phi_0}}(x, \xi)| }{ |\xi|}= c 
%
 \iff \frac{|p_{\theta}(x, \xi)|}{|\xi|} = c \left| \frac{\partial}{\partial \theta} \right|_{H_{\phi_0}}\Big\}.\end{equation}}
 In particular if $c=1$, then $(x, \xi) \in T^c_{H_{\phi_0}} \iff \xi = |\xi| u_{\theta}^*. $ As \eqref{Tc} shows, the parameter $c$ is
 not the usual ratio $\frac{p_{\theta}(x, \xi)}{|\xi|}$ of components off the moment map, because we choose the operator on $H$ to
 be $\sqrt{\Delta_H}$ rather than $\frac{\partial}{\partial \theta}.$

\subsubsection{\label{SPTH} Spectral theory}
Let $Y_{\ell}^m$ be the standard orthonormal basis of
joint eigenfunctions of $\Delta$ and of the generator $\frac{\partial}{\partial \theta}$ of rotations around the third axis.  Thus, $Y_{\ell}^m$ changes
by the phase $e^{i m \theta}$ under a rotation of angle $\theta$. The orthonormal eigenfunctions of $H_{\phi_0}$ are given by
$\psi_m(\theta) = C_{\phi_0} e^{im \theta}$ where $C_{\phi_0} = \frac{1}{L(H_{\phi_0})}$. Hence, the Fourier coefficients \eqref{FCDEF}
are constant multiplies of the Fourier coefficients relative to $\{e^{i m \theta}\}$. 
It follows that the  $m$th Fourier coefficient of $Y_{\ell}^m$  is its only non-zero Fourier coefficient along any lattitude
circle $H_{\phi_0}$, and that  $|\int_{H_{\phi_0}} Y_{\ell}^m e^{-im \theta} d\theta|^2 = ||Y_{\ell}^m||^2_{H_c}$. This is an example
where one can obtain estimates on individual Fourier coefficients of individual eigenfunctions. The situation is much more complicated
on higher dimensional spheres $\Ss^n$ when $H$ is a latitude `sub-sphere' $\Ss^d$ (see \cite{WXZ20}).

\subsubsection{Ladders and cones} 

Let $ \Lambda_{a} = \pcal^{-1}(a, 1) \subset S^* \Ss^2$ be the level set $\{p_{\theta} = a\}$. It is a Lagrangian torus
when $a \not= \pm 1$ and is the equatorial (phase space) geodesic when $a = \pm 1$.
A ray or ladder in the image of the moment map $\pcal$  is defined by $\{(m, E): \frac{m}{E} = a\} \subset \R^2_+$, and its  inverse
image under $\pcal$ is $\R_+ \Lambda_{a}  \subset T^* \Ss^2$.

 When $H$ is a latitude circle, then the wedge \eqref{CWEDGE}
is  a wedge around a ray in the image of the moment map, since the rays $\{\frac{p_{\theta}(x, \xi)}{|\xi|} = a\}$ and $\{ \frac{|\pi_{H_{\phi_0}}(x, \xi)| }{ |\xi|}= c\}$ are related by the constant $|\frac{\partial}{\partial \theta}|$. 

On the quantum level, a ray corresponds to a  `ladder'   $\{Y_{\ell}^m\}_{\frac{m}{\ell} = a}$ of eigenfunctions. The possible Weyl-Kuznecov sum formulae for latitude circles $H = H_{\phi_0}$ thus depend on the
two parameters $(\phi_0,  \frac{m}{\ell})$. The first corresponds to a latitude circle, the second to a ladder in the joint spectrum. It is better
to parametrize the ladder as $\frac{\mu_m}{\ell} = c $ as discussed above. 

\subsubsection{Caustic sequences and Gaussiam beams} There are special scenarios where the ladder of eigenfunctions
corresponds to  the  Lagrangian  torus $\{p_{\theta} = a |\xi| \} = \Lambda_a$ in $T^*\Ss^2$ and where $H_{\phi_0}$ is the caustic
of this Lagrangian torus, i.e.  a boundary component of its projection. 
 In this case, the  Fourier coefficients of the subsequence of $\{Y^{m}_{\ell}\}$  blow up at the rate $\ell^{1/6}$. This case is outside
 the scope of this article because the corresponding value of $c$ equals $1$. It will be addressed in a later article.
 
 Another extremal scenario is where $a = \pm 1$, i.e. the classical ray occurs on the boundary of the moment map image. The corresponding ladder of eigenfunctions consists of the Gaussian beams,
$ C_0 N^{\frac{1}{4} } (x_1 + i x_2)^N$,  around the equator $\gamma$. In Fermi-normal coordinates, they have the form $N^{(1)/4} e^{i N s} e^{- N y^2/2} $, where $s$ is arc-length along $\gamma$ and $y$ is the normal coordinate. This ladder again corresponds to $c=1$
and is outside the scope of this article; general examples with $c=1$ and $H$  totally geodesic  are described in \cite{WXZ20}.

\subsubsection{$H$ is a closed geodesic of $\Ss^2$ and $c<1$}
Suppose that   $0 < c < 1$ and that  $H$ is a closed geodesic. If $H$ is the equator, then it lies in the interior of the image of
the projection of the torus $\Lambda_a$ and the  unique non-zero
Fourier coefficients of the spherical harmonics $Y^m_{\ell}$ with $\frac{m}{\ell} \simeq c < 1$  uniformly
bounded above.  

On the other hand, one might restrict $Y^m_{\ell}$ to a meridian geodesic, in which case all the Fourier coefficients in the
range $[-\ell, \ell]$ can be non-zero. This is a $c < 1$ case to which our results apply. Note that when $m =0$ the Fourier coefficients
are those of the Legendre function $P_{\ell}(\cos \phi)$.

\subsubsection{Convex surface of revolution in $\R^3$}
All of the above  remarks  generalize to any convex surface of revolution, with  the equator defined as the unique rotationally invariant geodesic. There exist
zonal eigenfunctions and Gaussian beams along equators of general convex surfaces of revolution, so the
orders of magnitude and the eigenfunctions are of the same type.  We refer to \cite{Geis} for a recent study of how the restricted
$L^2$ norms vary with $c$.


\subsection{An Example on the Torus}

Here we verify Theorem \ref{main cone} for an easy example on the torus. This is in part to check the constant in Theorem \ref{main cone} and hence the constant in Theorem \ref{main FIO} with a direct computation. Though we are careful to track all of the dimensional constants in the computations, we find it prudent to verify the result directly.

Let $M = \T^n = \R^n/2\pi \Z^n$ and $H = \T^d$ be embedded in $M$ as the coordinate plane $T^d \times 0 \subset M$. We select the standard bases of exponentials
\[
	e_j(z) = (2\pi)^{-n/2} e^{i\langle z, j \rangle} \qquad \text{ and } \qquad \psi_k(x) = (2\pi)^{-d/2} e^{i\langle x, k \rangle} 
\]
indexed by $j \in \Z^n$ and $k \in \Z^d$, respectively. Note,
\begin{align*}
	\int_H e_j \overline \psi_k \, dV_H &= (2\pi)^{-(n+d)/2} \int_{\T^d} e^{i\langle x, j' - k \rangle} \, dx \\
	&= 
	\begin{cases}
		(2\pi)^{-(n-d)/2} & k + j' = 0 \\
		0 & k - j' \neq 0.
	\end{cases}
\end{align*}
where here $j' = (j_1,\ldots, j_d)$ is the first $d$ coordinates of $j$. The sum in Theorem \ref{main cone} is then
\[
	(2\pi)^{-(n-d)} \#\{ j \in \Z^n : |j| \leq \lambda, \ |j'|/|j| \in [a,b] \}.
\]

\begin{proposition} \label{torus computation}
	Let $M = \T^n$ and let $H = \T^d \times 0$ be a coordinate plane as above. If $[a,b] \subset (0,1)$, the sum in Theorem \ref{main cone} is
	\[
		(2\pi)^{-(n-d)} \frac{(\vol S^{d-1}) (\vol S^{n-d-1})}{n} \left( \int_a^b t^{d-1} (1 - t^2)^{\frac{n-d-2}{2}} \, dt \right) \lambda^n + O(\lambda^{n-1})
	\]
	by direct computation. We recall $\vol H = (2\pi)^d$ and see this agrees with Theorem \ref{main cone}.
\end{proposition}

\begin{proof}
	By counting cubes, the sum is
    \[
    	(2\pi)^{-(n-d)} |\{ (\xi \in \R^n : |\xi| \leq \lambda, \ |\xi'|/|\xi| \in [a,b] \}| + O(\lambda^{n-1})
    \]
where here the absolute value notation around the set denotes Lebesgue measure in $\R^n$. We parametrize this set by the map
\[
	\Phi(r,t,\omega,\eta) = (rt \omega, r \sqrt{1 - t^2} \eta)
\]
where $r \in [0,\lambda]$, $t \in [a,b]$, $\omega \in S^{d-1}$ and $\eta \in S^{n-d-1}$. The pullback of the Euclidean metric has the form
\[
	g(r,t,\omega,\eta) = \begin{bmatrix}
		1 & 0 & 0 & 0 \\
		0 & \frac{r^2}{1 - t^2} & 0 & 0 \\
		0 & 0 & r^2 t^2 g_{S^{d-1}}(\omega) & 0 \\
		0 & 0 & 0 & r^2(1 - t^2) g_{S^{n-d-1}}(\eta)
	\end{bmatrix},
\]
and hence the pullback of the Euclidean volume density is
\[
	|\det g(r,t,\omega,\eta)|^{1/2} \\= r^{n-1} t^{d-1} (1 - t^2)^{\frac{n-d-2}{2}} |\det g_{S^{d-1}}(\omega)|^{1/2} |\det g_{S^{n-d-1}}(\eta)|^{1/2}.
\]
Integrating yields
\[
	|\{ (\xi \in \R^n : |\xi| \leq \lambda, \ |\xi'|/|\xi| \in [a,b] \}| = \frac{(\vol S^{d-1})(\vol S^{n-d-1})}{n} \lambda^n \int_a^b t^{d-1} (1 - t^2)^{\frac{n-d-2}{2}} \, dt.
\]
The proposition follows.
\end{proof}

A similar computation can be carried out to verify Theorem \ref{main ladder} on the torus, in which we sum the norm-squares of the Fourier coefficients over the thickening strip
\[
	||k| - c|j|| \leq w(\lambda), \qquad |j| \leq \lambda 
\]
with slope $c$ and width $w(\lambda)$. More importantly, this example shows we cannot obtain asymptotics for a sum over a strip of constant width $w$, at least not without adding some more hypotheses. By the reductions before Proposition \ref{torus computation}, this ladder sum is exactly
\[
	(2\pi)^{-(n-d)} \#\{j \in \Z^n : |j| \leq \lambda, \ ||j'| - c|j|| \leq w\}.
\]
The set of $\xi \in \R^n$ with $||\xi'| - c|\xi|| \leq w$ is asymptotic to a $w\sqrt{1 - c^2}$-thickening of the cone $|\xi'| = c|\xi|$ of `slope' $c/\sqrt{1 - c^2}$.

In the case $n = 2$ and $d = 1$, $c$ may be taken so that the slope $c/\sqrt{1 - c^2}$ is rational and $w$ may be taken small enough to only include those lattice points in $\Z^2$ lying along the line of slope $c/\sqrt{1 - c^2}$. Any small, nonzero change in $w$ will not perturb the ladder sum of Theorem \ref{main ladder}, yet would be felt by the main term
\[
	\frac1\pi \frac{w}{\sqrt{1 - c^2}} \lambda.
\]
We conclude the remainder must be just as large as the main term. It is also possible, using a more careful computation, to locate jumps in this ladder sum of order $\lambda$ given a change in $w$ on the order of $\lambda^{-1}$.


\section{Proofs of Theorems \ref{main cone}, \ref{main ladder}, and Corollary \ref{main corollary}} \label{COROLLARY}

\subsection{Proofs of Theorems \ref{main cone} and \ref{main ladder}}

We begin by proving Theorem \ref{main cone}. We take the cone
\[
	\Omega_\lambda = \{(\mu',\lambda') \in \R^2 : \mu',\lambda' > 0, \ \lambda' \leq \lambda,  \text{ and } \mu'/\lambda' \in [a,b] \}.
\]
and realize the sum of Theorem \ref{main cone} is precisely $N(\Omega_\lambda)$ with $N$ from Theorem \ref{main FIO}. Take a smooth cutoff $\chi \in C^\infty(\R,[0,1])$ which takes the value $1$ on a neighborhood of $[a,b]$ and the value $0$ on a neighborhood of the complement of $(0,1)$, and set
\begin{equation}\label{tilde N}
	\tilde N = \sum_{j,k} \left| \int_H e_j \overline \psi_k \, dV_H \right|^2 \chi(\mu_k/\lambda_j) \delta_{(\mu_k,\lambda_j)}.
\end{equation}
Note $N(\Omega_\lambda) = \tilde N(\Omega_\lambda)$, hence we only need to show $\tilde N(\Omega_\lambda)$ satisfies the asymptotics of Theorem \ref{main cone}.

Fix a nonnegative Schwartz-class function $\rho$ with sufficiently small Fourier support and $\widehat \rho(0) = 1$. By construction,
\[
	|(N - \tilde N) * \rho(\mu,\lambda)| = O_{\rho,[a,b]}(\lambda^{-\infty}) \qquad \text{ for } \mu,\lambda > 0 \text{ and } \mu/\lambda \in [a,b].
\]
This and Theorem \ref{main FIO} yields $\tilde N * \rho(\mu,\lambda)$ is bounded by some constant times the order function
\[
	(1 + |(\mu,\lambda)|)^{n-2}.
\]
Theorem \ref{main cone} then follows by Theorem \ref{main tauberian}, the asymptotics in Theorem \ref{main FIO}, and an elementary computation.

Next, we prove Theorem \ref{main ladder}. Fix $[a,b] \subset (0,1)$ such that $a < c < b$ and consider the strip
\[
	S_\lambda = \{(\mu',\lambda') \in \R^2 : 0 < \lambda' \leq \lambda \text{ and } |\mu' - c\lambda'| \leq w(\lambda)\}.
\]
Let $W_r = \{(\mu',\lambda') \in \R^2 : 0 \leq \lambda' \leq r\}$ denote the strip of width $r$ along the first axis, and fix a constant $C$ for which
\[
	S_\lambda \setminus W_{Cw(\lambda)} \subset \{(\mu',\lambda') : \mu'/\lambda' \in [a,b] \}.
\]
We have by the basis property of $\{\psi_k\}$, the local Weyl law, and the hypothesis $w(\lambda) = O(\lambda^{1/2})$,
\[
	N(S_\lambda \cap W_{Cw(\lambda)}) = \sum_{\substack{j,k : \lambda_j \leq Cw(\lambda) \\ |\mu_k - c\lambda_j| \leq w(\lambda)}} \left| \int_H e_j \overline \psi_k \, dV_H \right|^2 \leq \sum_{j : \lambda_j \leq Cw(\lambda)} \int_H |e_j|^2 \, dV_H = O(\lambda^{n-1}).
\]
All that is left is to show $N(S_\lambda \setminus W_{Cw(\lambda)})$ satisfies the asymptotics of Theorem \ref{main ladder}, i.e.
\[
	N(S_\lambda \setminus W_{Cw(\lambda)}) = C_{H,M} c^{d-1} (1 - c^2)^{\frac{n-d-2}{2}} w(\lambda) \lambda^{n - 1} + O(\lambda^{n-1}),
\]
where we allow the constants implicit in the big-$O$ remainder to depend on $c$ and $w$. We replace $N$ with $\tilde N$ as in \eqref{tilde N} and have
\[
	N(S_\lambda \setminus W_{Cw(\lambda)}) = \tilde N(S_\lambda \setminus W_{Cw(\lambda)})
\]
by construction. Let $\rho$ be as before so that $\tilde N * \rho$ is again bounded by the same order function. Since $w(\lambda) \leq \lambda$ for $\lambda$ large
\[
	\int_{\partial^{[-1,1]} (S_\lambda \setminus W_{Cw(\lambda)})} \rho * \tilde N(\mu',\lambda') \, d(\mu',\lambda') = O(\lambda^{n-1}).
\]
By Theorem \ref{main tauberian},
\begin{align*}
	\tilde N(S_\lambda \setminus W_{Cw(\lambda)}) &= \int_{S_\lambda \setminus W_{Cw(\lambda)}} \rho * \tilde N(\mu',\lambda') \, d(\mu', \lambda') + O(\lambda^{n-1}) \\
	&= \int_{Cw(\lambda)}^\lambda \int_{-w(\lambda)}^{w(\lambda)} \rho * \tilde N(c\lambda' + t, \lambda') \, dt \, d\lambda' + O(\lambda^{n-1}).
\end{align*}
By Theorem \ref{main FIO}, the mean value theorem, and our hypotheses on the growth of $w(\lambda)$, we have
\[
	\rho * \tilde N(c \lambda' + t, \lambda') = \rho * \tilde N(c\lambda', \lambda') + O(w(\lambda)\lambda'^{n-3}) \qquad \text{ for all } |t| \leq w(\lambda).
\]
Continuing with our estimates, we have by $w(\lambda) = o(\lambda)$,
\[
	\int_{Cw(\lambda)}^\lambda \int_{-w(\lambda)}^{w(\lambda)} \rho * \tilde N(c\lambda' + t, \lambda') \, dt \, d\lambda' = 2w(\lambda) \int_0^\lambda \rho * \tilde N(c\lambda',\lambda') \, d\lambda' + O(\lambda^{n-1} + w(\lambda)^2 \lambda^{n-2}).
\]
By hypothesis, $O(\lambda^{n-1} + w(\lambda)^2 \lambda^{n-2}) = O(\lambda^{n-1})$, and hence
the right side reads,
\begin{multline*}
	2C_{H,M} w(\lambda) c^{d-1} (1 - c^2)^{\frac{n-d-2}{2}} \int_0^\lambda \lambda'^{n-2} \, d\lambda' + O(\lambda^{n-1}) \\
	= \frac{2C_{H,M}}{n - 1} w(\lambda) c^{d-1} (1 - c^2)^{\frac{n-d-2}{2}} \lambda^{n-1} + O(\lambda^{n-1})
\end{multline*}
by Theorem \ref{main FIO}. This concludes the proof of Theorem \ref{main ladder}.

\subsection{Proof of Corollary \ref{main corollary}}

For each $j$, set
\[
	\nu_j = \sum_k \left| \int_H e_j \overline \psi_k \, dV_H \right|^2 \delta_{\mu_k/\lambda_j}
\]
so that we write
\[
	\nu = \lim_{\lambda \to \infty} \lambda^{-n} \sum_{\lambda_j \leq \lambda} \nu_j.
\]
Let $f$ be a continuous function with support contained in $(0,1)$. Approximating $f$ above and below by step functions and applying Theorem \ref{main cone} on each constant component, we obtain
\begin{equation*}\label{main cone corollary 1}
	\lim_{\lambda \to \infty} \frac{n}{C_{H,M}\lambda^n} \sum_{\lambda_j \leq \lambda} \int f \, d\nu_j = \int_0^1 f(t) t^{d-1} (1 - t^2)^{\frac{n-d-2}{2}} \, dt
\end{equation*}
We now argue that this limit holds for bounded continuous functions on all of $\R$.

For $\delta > 0$, let $\chi_\delta$ be a continuous cutoff function which takes values in $[0,1]$, which is supported in $(0,1)$, and with $\chi_\delta = 1$ on $[\delta,1-\delta]$. For any bounded continuous function $f$ on $\R$, we have
\[
	\frac{n}{C_{H,M}\lambda^n} \sum_{\lambda_j \leq \lambda} \int f \chi_\delta \, d\nu_j = \int_0^1 f(t) \chi_\delta(t) t^{d-1} (1 - t^2)^\frac{n-d-2}{2} \, dt + o_\delta(1).
\]
We select $\delta(\lambda)$ decreasing to $0$ as $\lambda \to \infty$ slowly enough so that
\[
	\lim_{\lambda \to \infty} \frac{n}{C_{H,M}\lambda^n} \sum_{\lambda_j \leq \lambda} \int f \chi_{\delta(\lambda)} \, d\nu_j = \int_0^1 f(t) t^{d-1} (1 - t^2)^\frac{n-d-2}{2} \, dt.
\]
We now show the limit of the discrepancy
\[
	\lim_{\lambda \to \infty} \frac{n}{C_{H,M}\lambda^n} \sum_{\lambda_j \leq \lambda} \int f (1 - \chi_{\delta(\lambda)}) \, d\nu_j
\]
vanishes, after perhaps taking $\delta(\lambda) \to 0$ more slowly. By the triangle inequality, we have
\[
	\left| \int f ( 1 - \chi_{\delta(\lambda)}) \, d\nu_j \right| \leq \| f\|_{L^\infty(\R)} \int (1 - \chi_{[\delta(\lambda), 1 - \delta(\lambda)]}) \, d\nu_j
\]
where here $\chi_{[\delta(\lambda), 1 - \delta(\lambda)]}$ denotes the characteristic function of the interval $[\delta(\lambda), 1 - \delta(\lambda)]$. The following lemma concludes the proof of Corollary \ref{main corollary}.

\begin{lemma}
	Fix $\delta > 0$ and let $\chi_{[\delta,1-\delta]}$ denote the characteristic function of the interval $[\delta,1-\delta]$. Then,
	\[
		\frac{n}{C_{H,M}\lambda^n} \sum_{\lambda_j \leq \lambda} \int (1 - \chi_{[\delta,1-\delta]}) \, d\nu_j \lesssim \delta^{1/2} + C_\delta \lambda^{-1},
	\]
	where the constant implicit in the $\lesssim$ notation depends only on $H$ and $M$, and where $C_\delta$ only depends on $H$, $M$, and $\delta$.
\end{lemma}

\begin{proof}
	Since the $\psi_k$'s form an orthonormal basis for $L^2(H)$,
	\[
		\int 1 \, d\nu_j = \sum_k \left| \int_H e_j \overline \psi_k \, dV_H \right|^2 = \int_H |e_j|^2 \, dV_H.
	\]
	Hence by the pointwise Weyl law \cite[Theorem 29.1.4]{HIV},
	\[
		\sum_{\lambda_j \leq \lambda} \int 1 \, d\nu_j = (2\pi)^{-n} \frac{(\vol H) (\vol S^{n-1})}{n} \lambda^n + O(\lambda^{n-1}),
	\]
	and so
	\begin{align*}
		\frac{n}{C_{H,M}\lambda^n} \sum_{\lambda_j \leq \lambda} \int 1 \, d\nu_j &= \frac{\vol S^{n-1}}{(\vol S^{d-1}) (\vol S^{n-d-1})} + O(\lambda^{-1}) \\
		&= \int_0^1 t^{d-1} (1 - t^2)^\frac{n-d-2}{2} \, dt + O(\lambda^{-1}).
	\end{align*}
	where the second line follows from a similar computation as in the proof of Proposition \ref{torus computation}. Theorem \ref{main cone} yields
	\begin{multline*}
		\frac{n}{C_{H,M}\lambda^n} \sum_{\lambda_j \leq \lambda} \int (1 - \chi_{[\delta,1-\delta]}) \, d\nu_j\\
		= \int_0^\delta t^{d-1}(1 - t^2)^\frac{n-d-2}{2} \, dt + \int_{1 - \delta}^1 t^{d-1}(1 - t^2)^\frac{n-d-2}{2} \, dt + O_\delta(\lambda^{-1}).
	\end{multline*}
	The lemma follows by $1 \leq d \leq n-1$ and an elementary estimate.
\end{proof}


\section{Proof of Theorem \ref{main FIO}} \label{FIO}

\subsection{The Setup}

Set $P_M = \sqrt{-\Delta_M}$ and $P_H = \sqrt{-\Delta_H}$. These are first order, self-adjoint, elliptic pseudodifferential operators on their respective manifolds with principal symbols
\[
	p_M(z,\zeta) = \left( \sum_{i,j} g_M^{ij}(z) \zeta_i \zeta_j \right)^{1/2} \quad \text{ and } \quad p_H(x,\xi) = \left( \sum_{i,j} g_H^{ij}(x) \xi_i \xi_j \right)^{1/2}.
\]
Here, $g_M$ and $g_H$ are the Riemannian metric tensors on $M$ and $H$, respectively.

Now,
\[
	N = \sum_{j,k} \left| \int_H e_j \overline \psi_k \, dV_H \right|^2 \delta_{(\mu_k,\lambda_j)},
\]
which is a joint spectral measure of the operators $P_H \otimes I$ and $I \otimes P_M$ on $H \times M$ weighted by the norm-squared Fourier coefficients. We will want to rewrite these weights using half-densities so that we can use the theory of FIOs. We change notation so that $\psi_k$ and $e_j$ are instead the eigendensities
\[
	\psi_k = \tilde \psi_k |dV_H|^{1/2} \qquad \text{ and } \qquad e_j = \tilde e_j |dV_M|^{1/2}
\]
where $\tilde \psi_k$ and $\tilde e_j$ are now the corresponding eigenfunctions. Let $i : H \to M$ be the embedding and let $\delta_i$ be the half-density distribution in $H \times M$ for which
\begin{equation}\label{def delta_i}
	(\delta_i, f) = \int_H \tilde f(x,i(x)) \, dV_H
\end{equation}
for smooth test half-densities $f = \tilde f |dV_H \, dV_M|^{1/2}$ on $H \times M$. Then, we write the Fourier coefficient as
\[
	\int_H \tilde e_j \overline{\tilde \psi}_k \, dV_H = (\delta_i, \overline \psi_k \otimes e_j).
\]

For technical reasons, we will need to insert a pseudodifferential cutoff. Thankfully, we can do this at minimal cost.

{ 
\begin{lemma} \label{pseudodifferential lemma} Let $\chi$ be a smooth function on $\R$ taking values in $[0,1]$ for which $\chi = 1$ on a neighborhood of $[a,b]$ and $\supp \chi \Subset (0,1)$. The operator $B$ acting on distributions over $H \times M$ defined spectrally by
 \[
	B(\overline \psi_k \otimes e_j) = \chi(\mu_k/\lambda_j) \overline\psi_k \otimes e_j\footnote{We take the convention that $	B(\overline \psi_k \otimes e_j)=0$ whenever $\lambda_j=0$.}
\]
is a real, self-adjoint, $0$th order pseudodifferential operator with principal symbol
\[
	\chi(p_H(x,\xi)/p_M(z,\zeta)) \qquad \text{ for } (x,\xi,z,\zeta) \in T^*(H \times M)
\]
and with essential support in
\[
	\{(x,\xi,z,\zeta) : p_H(x,\xi)/p_M(z,\zeta) \in \supp \chi \}.
\]
\end{lemma}

We defer the proof of the lemma until after our reduction, but not without a few words first. The operator $B$ is equal to $\chi(\frac{P_H}{P_M})$. Here, $\frac{P_H}{P_M}$ is not a pseudo-differential operator but $\chi(\frac{P_H}{P_M})$ is a zeroth order pseudo-differential operator due to the properties of the cutoff.

To put this statement into context, we recall  the notion of  polyhomogeneous symbols $S^m(\Gamma)$  relative to a choice of conic open subset $\Gamma \subset T^*M - \{0\}$; see \cite{HIV}, Volume 3, page 83. Namely, the standard symbol estimates of \cite[(28.1.1)']{HIV} are only assumed to be valid in $\Gamma$.  This notion is developed more systematically on manifolds in \cite[Page 86]{HFIO}.

Formally, the principal symbol of the quotient operator $\frac{P_H}{P_M}$ is $\frac{p_H}{p_M}$. This is not a symbol on $T^*(M \times H) -\{0\}$, and indeed it is not even defined on the sub-cone $0_M \times T^* H - \{0\}$. However, it is a  symbol on the cone $\Gamma \subset T^*(M \times H) \backslash 0$ where $p_H(x,\xi)/p_M(z,\zeta)\in \supp \chi$. Note that if $\rm{Supp} \chi = [a, b]$ then $p_H(x,\xi)/p_M(z,\zeta)\in \supp \chi$ if and only if $a p_M(z, \zeta) \leq p_H(x,\xi) \leq b p_M(z,\zeta) $. We may further assume that $p_H(x, \xi) + p_M(z, \eta) \geq \delta > 0$ for some $\delta > 0$, since the cutoff of $\chi(\frac{P_H}{P_M})$ to the complement is a smoothing operator. Then at least one of $p_H(x, \xi),p_M(z, \eta)$ is $\geq \delta/2$ and so both symbols are uniformly bounded above zero. Then the symbol $q(y,\eta, x, \xi) = \chi(p_H(x,\xi)/p_M(z,\zeta))$ is homogeneous of degree $0$ and elliptic in $\Gamma. $  Moreover, it is a smoothing operator on the complement of $\Gamma$.
}

We set
\begin{align*}
	N_B &= \sum_{j,k} \chi^2(\mu_k/\lambda_j) |(\delta_i, \overline \psi_k \otimes e_j)|^2 \delta_{(\mu_k,\lambda_j)} \\
	&= \sum_{j,k} |(\delta_i, B(\overline \psi_k \otimes e_j))|^2 \delta_{(\mu_k,\lambda_j)}.
\end{align*}
Similarly as in the proof of Theorem \ref{main cone},
\[
	|\rho * N(\mu,\lambda) - \rho * N_B(\mu,\lambda)| = O(\lambda^{-\infty}) \qquad \text{ for $\mu/\lambda \in [a,b]$,}
\]
so we may freely exchange $N$ for $N_B$ in the statement of the theorem.

Similar to \eqref{def delta_i}, if we denote $\delta_{i \times i}$ the half-density on $H\times H\times M\times M$ by
 \begin{equation}\label{def delta{ixi}}
 	(\delta_{i \times i}, f) = \int_H\int_H \tilde f(x,y,i(x),i(y)) \, dV_H(x)\,dV_H(y),
 \end{equation}
for smooth test half-densities $f = \tilde f |dV_H \, dV_H \, dV_M \, dV_M|^{1/2}$ on $H\times H\times M\times M$. 
Then we have
\begin{align*}
	|(\delta_i, B(\overline \psi_k \otimes e_j))|^2 &= \chi(\mu_k/\lambda_j)^2 |(\delta_i, \overline \psi_k \otimes e_j)|^2  \\
	&= \chi(\mu_k/\lambda_j)^2 (\delta_i \otimes \delta_i, \overline \psi_k \otimes e_j \otimes \overline{\overline \psi_k \otimes e_j}) \\
	&= ( \delta_{i \times i}, A(\overline \psi_k \otimes \psi_k \otimes e_j \otimes \overline e_j)) \\
	&= ( A\delta_{i \times i}, \overline \psi_k \otimes \psi_k \otimes e_j \otimes \overline e_j)
\end{align*}
where $A$ is a pseudodifferential operator on $H^2 \times M^2$ defined spectrally by
\begin{equation}\label{def A}
	A(\overline \psi_k \otimes \psi_{k'} \otimes e_j \otimes \overline e_{j'}) = \chi(\mu_k/\lambda_j) \chi(\mu_{k'}/\lambda_{j'}) \beta(\mu_k/\mu_{k'}) \beta(\lambda_j/\lambda_{j'})
\end{equation}
where $\beta \in C^\infty(\R, [0,1])$ is identically $1$ on a neighborhood of $1$ and has support in $(1/2,2)$. Similar to $B$, $A$ is a real, $0$th order, self-adjoint pseudodifferential operator. The $\beta$ cutoffs are there to ensure the symbol of $A$ is smooth near the axes. { Note, the third equality above holds here since $j' = j$ and $k' = k$.}

Using this and a Fourier transform, we have
\begin{align*}
	\widehat N_B(s,t) &= \sum_{j,k} ( A \delta_{i \times i}, \overline \psi_k \otimes \psi_k \otimes e_j \otimes \overline e_j) e^{-i(s\mu_k + t \lambda_j)} \\
	&= (A \delta_{i \times i}, \overline{e^{isP_H}} \otimes e^{-itP_M}),
\end{align*}
interpreted in a distributional sense. We let
\[
	U(s,t,x,y,z,w) = \overline{e^{isP_H}(x,y)} e^{-itP_M}(z,w)
\]
be the half-density distribution kernel of the tensored half-wave operators, and by an abuse of notation, we let $U$ denote the operator with the kernel above taking smooth half-densities on $H^2 \times M^2$ to half-density distributions on $\R^2$. Then, we have
\[
	\widehat N_B |ds \, dt|^{1/2} = U \circ A \circ \delta_{i \times i}.
\]

{
\begin{proof}[Proof of Lemma \ref{pseudodifferential lemma}] That $B$ is real and self-adjoint is clear from its definition. Furthermore, we may remove the complex conjugate over $\psi_k$ and write
\[
	B(\psi_k \otimes e_j) = \chi(\mu_k/\lambda_j) \psi_k \otimes e_j.
\]
This will slightly simplify the calculations to come. We must verify that it is a pseudodifferential operator with the indicated symbol and essential support. To this end, we write $B$ locally up to lower order terms. In what follows, we will write
\[
	b(\sigma, \tau) = \chi(\sigma/\tau),
\]
where here $b$ is positive-homogeneous of order $0$ and smooth on $\R^2 \setminus 0$ since $\chi$ is smooth and has compact support in the interval $(0,\infty)$. Note, we may declare $b(0,0) = 0$ and regularize $b$ near the origin at the cost of a smooth error.

We first note that
\[
	e^{is P_H \otimes it P_M} = e^{is P_H} \otimes e^{itP_M},
\]
since their evaluations on joint eigenfunctions $\psi_k \otimes e_j$ agree. By Fourier inversion we write
\[
	B = b(P_H, P_M) = (2\pi)^{-2} \iint_{\R^2} \widehat b(s,t) e^{is P_H} \otimes e^{it P_M} \, ds \, dt
\]
Let $\rho$ be a Schwartz-class function on $\R^2$ such that $\rho \equiv 1$ near the origin and $\rho \equiv 0$ outside of a neighborhood of the origin. Then, we cut the integral into $\rho(s,t)$ and $1 - \rho(s,t)$ parts. Note since $b$ is in class $S^0(\R^2)$, its Fourier transform has singular support at $0$ (see the proof of Theorem 4.3.1 in \cite{Sog17}). Hence, $\widehat b(s,t) (1 - \rho(s,t))$ is Schwartz-class, and hence an integration by parts reveals
\begin{multline*}
	(2\pi)^{-2} \iint_{\R^2} \widehat b(s,t) (1 - \rho(s,t)) e^{is P_H} \otimes e^{it P_M}(\psi_k \otimes e_j) \, ds \, dt \\
	= (2\pi)^{-2} \iint_{\R^2} \widehat b(s,t) (1 - \rho(s,t)) e^{is \mu_k} \otimes e^{it \lambda_j}(\psi_k \otimes e_j) \, ds \, dt = O(|(\mu_k, \lambda_j)|^{-\infty}).
\end{multline*}
It suffices now to show that
\begin{equation} \label{pseudodifferential lemma 1}
	(2\pi)^{-2} \iint_{\R^2} \widehat b(s,t) \rho(s,t) e^{is P_H} \otimes e^{it P_M} \, ds \, dt
\end{equation}
is the desired pseudodifferential operator.

Next, as in \cite{Sog17}, we use H\"ormander's small time parametrix for the half-wave operator to obtain
\begin{align*}
	&e^{isP_H} \otimes e^{is P_M}(x,z,y,w) \\
	&= e^{isP_H}(x,y) e^{itP_M}(z,w) \\
	&= (2\pi)^{-n-d} \int_{\R^n} \int_{\R^d} e^{i (s p_H(y,\xi) + \varphi_H(x,y,\xi) + t p_M(w,\zeta) + \varphi_M(z,w,\zeta))} q_H(s,x,y,\xi) q_M(t,z,w,\zeta) \, d\xi \, d\zeta
\end{align*}
modulo a smooth kernel, where
\[
	\varphi_H(x,y,\xi) = \langle x - y , \xi \rangle + O(|x - y|^2|\xi|) \quad \text{ and } \quad \varphi_M(z,w,\zeta) = \langle z - w, \zeta \rangle + O(|z - w|^2 |\zeta|)
\]
and $q_H$ and $q_M$ are zeroth-order symbols. Next, we examine the contribution of the integrals in $s$ and $t$ to \eqref{pseudodifferential lemma 1}, namely
\begin{multline*}
	\iint_{\R^2} e^{i(sp_H(y,\xi) + tp_M(w,\zeta))} \widehat b(s,t) \rho(s,t) q_H(s,x,y,\xi) q_M(t,z,w,\zeta) \, ds \, dt \\
	= \iiiint e^{i(s(p_H(y,\xi) - \sigma) + t(p_M(w,\zeta) - \tau))} b(\sigma, \tau) \rho(s,t) q_H(s,x,y,\xi) q_M(t,z,w,\zeta) \, ds \, dt \, d\sigma \, d\tau
\end{multline*}
We claim that (1) the integral above is a symbol $a(x,y,z,w,\xi,\zeta)$ in class $S^0$, and (2) it has rapid decay outside of $\supp b$, and (3) it has principal term $b(p_H(y,\xi), p_M(w,\zeta))$. We are done after verifying these three claims.

We start, in fact, with (2). We let $\beta$ be homogeneous of order $0$ on $\R^2$ which takes $\beta \equiv 1$ on $\supp b$ and $\beta \equiv 0$ near the axes of $\R^2$. We then cut the integral above by $\beta(p_H, p_M)$ and $1 - \beta(p_H, p_M)$ parts. A standard integration by parts argument shows the latter part is in $S^{-\infty}$. Hence, we consider the former part,
\[
	\iiiint e^{i(s(p_H - \sigma) + t(p_M - \tau))} b(\sigma, \tau) \beta(p_H, p_M) \rho(s,t) q_H(s,x,y,\xi) q_M(t,z,w,\zeta) \, ds \, dt \, d\sigma \, d\tau,
\]
which satisfies (2) trivially. 

Next, we establish (1). We perform a change of variables and write the integral above as
\[
	p_M^2 \iiiint e^{ip_M ( s(p_H/p_M - \sigma) + t(1 - \tau))} b(\sigma, \tau) \beta(p_H/p_M, 1) \rho(s,t) q_H(s,x,y,\xi) q_M(t,z,w,\zeta) \, ds \, dt \, d\sigma \, d\tau.
\]
Now the phase function vanishes at its critical point $(s,t,\sigma, \tau) = (0,0,p_H/p_M, 1)$, at which it is nondegenerate. By a routine stationary phase argument, the result is a symbol with principal term
\begin{multline*}
	(2\pi)^2 b(p_H/p_M, 1) \beta(p_H/p_M, 1) \rho(0,0) q_H(0, x, y, \xi) q_M(0, z, w, \zeta) \\
	= (2\pi)^2 b(p_H, p_M) q_H(0, x, y, \xi) q_M(0, z, w, \zeta).
\end{multline*}
Again, one must repeat this argument with derivatives in $x,y,z,w$ and in $\xi, \zeta$, where the order of the leading term decreases with each derivative on the latter variables. We leave the details to the reader with the following technical warning: When a derivative, for example $\partial_{y_{ i}}$, hits the oscillatory part, a factor such as $s \partial_{y_i} p_H/p_M$ comes down into the amplitude. This would be troubling if not for the presence of the cutoff $\beta$, which ensures the amplitude remains in class $S^0$. 

Finally, we move to (3). We have shown that the Schwartz kernel of $B$ can be written
\[
	(2\pi)^{-d-n} \int_{\R^d} \int_{\R^n} e^{i(\varphi_H(x,y,\xi) + \varphi_M(z,w,\zeta))} (2\pi)^{-2} a(x,y,z,w,\xi,\zeta) \, d\zeta \, d\xi
\]
where $a \in S^0$ satisfies
\[
	a(x,x,z,z,\xi,\zeta) = (2\pi)^2 b(p_H, p_M)
\]
modulo a lower-order term since $q_H(0,x,x,\xi) \equiv 1$ and $q_M(0,z,z,\zeta) \equiv 1$ (see \cite[Chapter 4]{Sog17}). Hence, by equivalence of phase functions (see \cite[Chapter 3]{Sog17}), $B$ has principal symbol $b(p_H, p_M)$, as desired. This concludes claim (3) and the proof of the lemma.
\end{proof}
}

\subsection{The Symbolic Data of the Parts}

Next, we compute the symbolic data of the Lagrangian (actually, conormal) distribution $\delta_{i \times i}$. Then, we show that $U \circ A$ is a Fourier integral operator and compute its symbolic data.

We will use the following notation. $I^m(X, \Lambda)$ will denote the space of Lagrangian distributions on $X$ of order $m$ associated to the conic Lagrangian $\Lambda \subset T^*X \setminus 0$. $I^m(X \times Y,\mathcal C')$ will then be used to denote the space of Fourier integral operators from $Y$ to $X$, of order $m$, associated with the canonical relation $\mathcal C \subset T^*X \setminus 0 \times T^*Y \setminus 0$. For each $x \in H$ we let $\pi_x : T_{ix}^*M \to T_x^*H$ be the pullback of covectors through $i$. We will usually suppress the subscripted base point in the notation, e.g. write $(x,\pi \zeta)$ instead of $(x,\pi_x \zeta)$. Finally, we recall $g_H$ and $g_M$ are the respective metric tensors on $H$ and $M$.

\begin{proposition}\label{delta symbolic data}
$\delta_{i \times i} \in I^{\codim H/2}(H^2 \times M^2, \Lambda)$ where
\[
	\Lambda = \{(x, \pi \zeta, y, -\pi \omega, ix, -\zeta, iy, \omega) :
	x,y \in H, \ \zeta \in T_{ix}^*M, \ \omega \in T_{iy}^*M, \ (\zeta,\omega) \neq 0 \},
\]
and has principal symbol equal to (modulo a Maslov factor) the transport of the half-density
\[
	(2\pi)^{-\codim H/2} \frac{|g_H(x)|^{1/4} |g_H(y)|^{1/4}}{|g_M(ix)|^{1/4}|g_M(iy)|^{1/4}} |dx \, dy \, d\zeta \, d\omega|^{1/2}
\]
via the implied parametrization of $\Lambda$ by $(x,y,\zeta,\omega)$.
\end{proposition}

This proposition is an application of the following lemma to the embedding $i \times i : H^2 \to M^2$, whose proof we defer until the end of this subsection.

\begin{lemma}\label{delta symbolic data}
	Let $i : H \to M$ be a smooth embedding and let $\Gamma_i = \{(x,ix) : x \in H\}$ denote its graph in $H \times M$. Then, $\delta_i$ defined by \eqref{def delta_i} is a conormal distribution in $I^{\codim H/4}(H \times M, N^* \Gamma_i \setminus 0)$ where
	\[
		N^* \Gamma_i := \{(x,\pi\zeta, ix, -\zeta) : x \in H, \ \zeta \in T_{ix}^*M\}
	\]
	is the conormal bundle of the graph $\Gamma_i$ of $i$. Moreover, $\delta_i$ has principal symbol equal (modulo a Maslov factor) to the transport of the half-density
	\[\textit{}
		(2\pi)^{-\codim H/4} \frac{|g_H(x)|^{1/4}}{|g_M(i(x))|^{1/4}} |dx \, d\zeta|^{1/2}
	\]
	via the implied parametrization of $N^*\Gamma_i$ by $(x,\zeta)$.
\end{lemma}

Next, we compute the wavefront relation of $U$. We will use $G_M^t$ to denote the time-$t$ homogeneous geodesic flow on $T^*M \setminus 0$, and similar for $G_H^s$ on $H$. Recall the half-wave kernel $e^{-itP_M}(z,w)$ as a half-density distribution on $\R \times M \times M$ is a Lagrangian distribution associated with the Lagrangian submanifold
\[
	\{ (t,\tau,z,\zeta,w,-\omega) : \tau + p_M(z,\zeta) = 0 , \ (z,\zeta) = G_M^t(w,\omega) \},
\]
or equivalently,
\begin{equation}\label{half-wave symbolic data}
	\{(t,-p_M(z,\zeta),z,\zeta, G_M^t(z,-\zeta)) : t \in \R, \ (z,\zeta) \in T^*M \setminus 0\},
\end{equation}
with principal symbol equal (modulo Maslov factors) to the transport of $(2\pi)^{1/4} |dt \, dz \, d\zeta|^{1/2}$ via the implied parametrization by $t \in \R$ and $(z,\zeta) \in T^*M\setminus 0$ (see \cite[\S 29.1]{HIV}). We realize $e^{-itP_M}(x,y)$ as the kernel of the Fourier integral operator $U_M \in I^{-1/4}(\R \times M^2, \mathcal C'_M)$ from $M^2$ to $\R$, with canonical relation
\[
	\mathcal C_M = \{(t,-p_M(z,\zeta); z,-\zeta, G_M^{-t}(z,\zeta) ) : t \in \R, \ (z,\zeta) \in T^*M \setminus 0 \}.
\]
Here, we have used that $(w,-\omega) = G_M^t(z,-\zeta)$ if and only if $(w,\omega) = G_M^{-t}(z,\zeta)$. Similarly, $\overline{e^{isP_H}(x,y)}$ is associated with the Lagrangian manifold
\[
	\{(s,-p_H(x,\xi), x, \xi, G_H^{-s}(x,-\xi)) : s \in \R, \ (x,\xi) \in T^*H \setminus 0\}
\]
and has principal symbol $(2\pi)^{1/4} |ds \, dx \, d\xi|^{1/2}$. Again, define $U_H \in I^{-1/4}(\R \times H^2, \mathcal C_H')$ as the operator with kernel $\overline{e^{isP_H}(x,y)}$ and canonical relation
\[
	\mathcal C_H = \{(s,-p_H(x,\xi); x,-\xi,G_H^s(x,\xi) ) : s \in \R, \ (x,\xi) \in T^*H \setminus 0 \}.
\]

By the calculus of wavefront sets, the tensored operator $U_H \otimes U_M$ satisfies
\[
	\WF'(U_H \otimes U_M) \subset (\mathcal C_H \times \mathcal C_M) \cup (\mathcal C_H \times 0) \cup (0 \cup \mathcal C_M).
\]
Now, $U$ is precisely a permutation of the variables $U_H \otimes U_M$, and it is precisely the latter two components in the union above which prevent $U$ from being a Fourier integral operator. Composition of $U$ with our pseudodifferential cutoff $A$ in \eqref{def A} rescues us by excluding those elements
\[
	(s,\sigma,t,\tau; x,\xi, y,\eta, z, \zeta, w, \omega) \in \WF'(U)
\] 
for which $p_H(x,\xi)/p_M(z,\zeta) \not\in \supp \chi$ or $p_H(y,\eta)/p_M(w,\omega) \not\in \supp \chi$, which kills these problematic components. This cutoff also precludes any of $\xi,\eta,\zeta,\omega$ from vanishing. We then have the following. 

\begin{proposition}\label{U symbolic data} $U \circ A$ is a Fourier integral operator in $I^{-1/2}(\R^2 \times (H^2 \times M^2), \mathcal C')$ with canonical relation
\begin{multline*}
	\mathcal C = \{(s,-p_H(x,\xi), t, -p_M(z,\zeta); x, -\xi, G_H^s(x,\xi), z, -\zeta, G_M^{-t}(z,\zeta)) : \\
	s,t \in \R, \ (x,\xi) \in T^*H \setminus 0, \ (z,\zeta) \in T^*M \setminus 0\}
\end{multline*}
and principal symbol (modulo a Maslov factor) equal to the transport of
\[
	(2\pi)^{1/2} \chi( p_H(x,\xi)/p_M(z,\zeta) )^2 |ds \, dt \, dx \, d\xi \, dz \, d\zeta|^{1/2}
\]
via the implied parametrization.
\end{proposition}

This is an immediate application of the composition formula for a Fourier integral operator with a pseudodifferential operator. To simplify the resulting canonical relation, we have used that the symbols $p_H$ and $p_M$ remain constant along their respective Hamiltonian (read: geodesic) flows, and that both symbols are even.

\begin{proof}[Proof of Lemma \ref{delta symbolic data}]
	We select local coordinates $(y_1,\ldots,y_n)$ for a neighborhood of a point on $H$ for which $y_{d+1} = \cdots = y_n = 0$ parametrizes $H$. The immersion
	\[
		(x_1,\ldots,x_d) \mapsto (x_1,\ldots,x_d,0,\ldots,0)
	\]
	parametrizes $H$ in these local coordinates. Together, $(x,y) \in \R^d \times \R^n$ parametrize $H \times M$ in a neighborhood of a point on the graph of $i$. We let $f(x,y) = \tilde f(x,y)|g_H(x)|^{1/4}|g_M(y)|^{1/4}|dx \, dy|^{1/2}$ be a smooth test half-density on $H \times M$ supported in this neighborhood. By \eqref{def delta_i} and Fourier inversion, we write
	\begin{align*}
		(\delta_i, f) &= \int_{\R^d} \tilde f(x,(x,0)) |g_H(x)|^{1/2} \, dx \\
		&= (2\pi)^{-n} \int_{\R^n} \int_{\R^n} \int_{\R^d} e^{i\langle y - (x,0) , \zeta \rangle} \tilde f(x,y) |g_H(x)|^{1/2} \, dx \, dy \, d\zeta \\
		&= (2\pi)^{-n} \int_{\R^n} \int_{\R^n} \int_{\R^d} e^{i\langle y - (x,0) , \zeta \rangle} f(x,y) \frac{|g_H(x)|^{1/4}}{|g_M(y)|^{1/4}} \, |dx \, dy|^{1/2} \, d\zeta,
	\end{align*}
	and hence we write
	\[
		\delta_i(x,y) 
		= (2\pi)^{-n + \codim H/4} \left( \int_{\R^n} e^{i\langle y - (x,0),\zeta \rangle} (2\pi)^{-\codim H/4} \frac{|g_H(x)|^{1/4}}{|g_M(y)|^{1/4}} \, d\zeta \right) |dx \, dy|^{1/2}
	\]
	in oscillatory form. We set $\varphi(x,y,\zeta) = \langle y - (x,0), \zeta \rangle$ and note it is a nondegenerate phase function with critical set
	\[
		C_\varphi = \{ (x,(x,0),\zeta) : x \in \R^d, \ \zeta \in \R^n \setminus 0 \}.
	\]
	We let $\zeta' = (\zeta_1,\ldots, \zeta_d)$ be the first $d$ coordinates of $\zeta$. Via the map $(x,y,\zeta) \mapsto (x,y,\varphi'_\zeta)$, $C_\varphi$ parametrizes
	\[
		\{(x,(x,0),-\zeta',\zeta) : x \in \R^d, \ \zeta \in \R^n \setminus 0 \}
	\]
	which is precisely $N^*\Gamma_i \setminus 0$ in canonical local coordinates of $T^*(H \times M)$. We conclude
	\[
		\delta_i \in I^{\codim H/4}(H \times M, N^*\Gamma_i \setminus 0).
	\]
	
	Finally we compute the invariant half-density on $N^*\Gamma_i\setminus 0$. We parametrize $C_\varphi$ by $x$ and $\zeta$ as indicated above and have Leray density
	\begin{align*}
		d_\varphi &= \left|\frac{\partial(x,\theta,\varphi_\zeta')}{\partial(x,y,\zeta)}\right|^{-1} |dx \, d\zeta| \\
		&= \left|\frac{\partial(x,\zeta, y - (x,0))}{\partial(x,y,\zeta)} \right|^{-1} |dx \, d\zeta| \\
		&= \left| \det \begin{bmatrix}
			I & 0 & 0 \\
			0 & 0 & I \\
			* & I & 0
		\end{bmatrix} \right|^{-1} |dx \, d\zeta| \\
		&= |dx \, d\zeta|.
	\end{align*}
	The invariant homogeneous half-density on $N^*\Gamma_i \setminus 0$ is then the transport of 
	\[
		(2\pi)^{-\codim H/4} \frac{|g_H(x)|^{1/4}}{|g_M(y)|^{1/4}} |dx \, d\zeta|^{1/2}
	\]
	to $N^*\Gamma_i \setminus 0$ via its parametrization by $x$ and $\zeta$. This completes the proof.
\end{proof}

\subsection{Cleanness of the Composition} Before computing the principal symbol of the composition $U \circ A \circ \delta_{i \times i}$, we must first check the composition is clean.

Consider the projection $\mathcal C \to T^*(H^2 \times M^2)$ and the inclusion $\Lambda \to T^*(H^2 \times M^2)$. The fiber product of these maps consists of a universal object $F$ and two maps $F \to \mathcal C$ and $F \to \Lambda$ for which the diagram 
\begin{equation}\label{fiber product diagram}
	\begin{tikzcd}
		F \arrow[r] \arrow[d] & \mathcal C \arrow[d] \\
		\Lambda \arrow[r] & T^*(H^2 \times M^2)
	\end{tikzcd}
\end{equation}
In this case, \eqref{fiber product diagram} is said to be a fiber product diagram. The composition $\mathcal C \circ \Lambda$ may be realized as the image of the composition $F \to \mathcal C \to T^*(\R^2)$. In our specific case, the fiber product induces relationships
\[
	z = ix, \ \xi = -\pi \zeta, \ (y,-\pi \omega) = G_H^s(x,\xi), \text{ and } (iy, \omega) = G_M^{-t}(z,\zeta)
\]
amongst the elements of $\Lambda$ and $\mathcal C$ as they appear in Propositions \ref{delta symbolic data} and \ref{U symbolic data}, respectively, and have
\begin{multline}\label{def F}
	F = \{(s,t,x,\zeta) : s,t \in \R, \ x \in H, \ \zeta \in T_x^*M \setminus 0, \\ (i \otimes I)G_H^s(x,\pi \zeta) = (I \otimes \pi) G_M^{t}(ix, \zeta) , \ \frac{p_H(x,\pi\zeta)}{p_M(ix,\zeta)} \in \supp \chi \}.
\end{multline}
Here, we have simplified the set using the change of variables $\zeta \mapsto -\zeta$ and the identity $(I \otimes -I) G_M^{-t}(z,-\zeta) = G_M^t(z,\zeta)$. The projection $F \to \mathcal C \circ \Lambda$ is then
\[
	(s,t,x,\zeta) \longmapsto (s,-p_H(x,\pi\zeta), t, -p_M(ix,\zeta)).
\]

\begin{remark} The composition $\mathcal C \circ \Lambda$ indicates that, to obtain finer estimates on the joint asymptotics of Fourier coefficients, we need hypotheses which constrain the size of the set of points $(s,t,x,\zeta)$ for which
\[
	(i \otimes I) G_H^s(x,\pi \zeta) = (I \otimes \pi) G_M^t(ix,\zeta).
\]
Geometrically, such a point corresponds to a configuration of two geodesic segments, one each in $H$ and $M$, which meet at their endpoints so that the velocity of the one in $M$ coincides with the velocity of the one in $H$ after an orthogonal projection. Such refinements are the subject of ongoing work (see \cite{WXZ20}).
\end{remark}

By the composition of $\mathcal C$ and $\Lambda$ being clean, we mean that the fiber product diagram \eqref{fiber product diagram} is clean. In particular, the diagram \eqref{fiber product diagram} holds in the category of smooth manifolds, and for each point $p \in F$ and its corresponding images $(a;b) \in \mathcal C$ and $b \in \Lambda$, the linearized diagram
\begin{equation}\label{linearized diagram}
	\begin{tikzcd}
		T_p F \arrow[r] \arrow[d] & T_{(a;b)} \mathcal C \arrow[d] \\
		T_b \Lambda \arrow[r] & T_{b} T^*(H^2 \times M^2)
	\end{tikzcd}
\end{equation}
is also a fiber product diagram. In this case, the excess of the diagram is
\begin{align*}
	e &= (\dim T_p F + \dim T_b T^*(H^2 \times M^2)) - (\dim T_b \Lambda + \dim T_{(a;b)} \mathcal C) \\
	&= \dim T_p F - 2
\end{align*}
and is constant on connected components of $F$.

In our situation, the fiber product diagram is clean if and only if the linearization of the relation
\[
	(i \otimes I)G_H^s(x,\pi \zeta) = (I \otimes \pi) G_M^{t}(ix, \zeta)
\]
from \eqref{def F} defines the tangent space of $F$. We introduce a little notation to make this precise. We let $i$ and $\pi$ denote, as an abuse of notation, their respective linearizations. We let $H_{p_H}$ and $H_{p_M}$ denote the Hamilton vector fields associated with symbols $p_H$ and $p_M$, respectively, and recall
\[
	\frac{d}{ds} G_H^s = H_{p_H} \qquad \text{ and } \qquad \frac{d}{dt} G_M^t = H_{p_M}.
\]
We will also use the prime notation to indicate the (coefficients of the) vector associated with a variable, e.g. $(s',t',x',\zeta') \in T_{(s,t,x,\zeta)}(\R^2 \times T^*_HM)$. The linearization of the relation above is then written
\begin{equation}\label{linearized relation}
	(i \otimes I)(s'H_{p_H} + dG_H^s(x',\pi\zeta')) = (I\otimes \pi)(t'H_{p_M} + dG_M^t(ix',\zeta')),
\end{equation}
and one quickly checks that the cleanness condition is equivalent to:
\begin{equation}\label{simple clean}
	``\text{If $(s',t',x',\zeta')$ satisfies \eqref{linearized relation}, then $(s',t',x',\zeta') \in T_pF$}"
\end{equation}
at each $p \in F$.

We now verify that the composition of $\mathcal C$ and $\Lambda$ is always clean provided the $(s,t)$ coordinates of $\mathcal C$ lie in a suitably small neighborhood of the origin. We let
\[
		F_0 = \{(0,0,x,\zeta) : x \in H, \ \zeta \in T_x^*M, \ \frac{p_H(x,\pi\zeta)}{p_M(ix,\zeta)} \in \supp \chi \}
\]
be the component of $F$ which lies at $(s,t) = (0,0)$. Furthermore for an open neighborhood $\mathcal O$ of the origin in $\R^2$, we let $\mathcal C_{\mathcal O}$ denote the intersection of $\mathcal C$ with $T^*(\mathcal O \times H^2 \times M^2)$.

\begin{lemma}\label{isolated at origin}
	There exists an open neighborhood $\mathcal O$ of the origin in $\R^2$ such that $F_0 = F \cap (\mathcal O \times T^*_H M)$. As a consequence,
	\[
		\mathcal C_{\mathcal O} \circ \Lambda = \{(0,\sigma,0,\tau) \in T_0^*\R^2 : \sigma, \tau < 0, \ \sigma/\tau \in \supp \chi \}.
	\]
\end{lemma}

Since $\mathcal C_{\mathcal O} \circ \Lambda$ is the only part of the whole composition $\mathcal C \circ \Lambda$ which contributes to the set on the right in the lemma, the symbolic data of $U \circ A \circ \delta_{i \times i}$ at the origin is determined by the composition $\mathcal C_{\mathcal O} \circ \Lambda$ and the relevant calculus. Now we verify that the composition is clean.

\begin{proposition}\label{clean at origin}
	The composition $\mathcal C_{\mathcal O} \circ \Lambda$ is clean with excess $d + n - 2$.
\end{proposition}

We start by relating the Hamilton vector fields $H_{p_M}$ and $H_{p_H}$ along the embedded manifold $H$ in a convenient choice of coordinates.
Let $(z_1,\ldots,z_n)$ be local coordinates for $M$ such that $(x_1,\ldots,x_d) \mapsto (x_1,\ldots,x_d,0,\ldots,0)$ parametrizes a neighborhood in $H$ and the metric $g_M$ along $H$ is a block matrix
\[
	g_M(x,0) = \begin{bmatrix}
		g_H(x) & 0 \\
		0 & h(x)
	\end{bmatrix}.
\]
Note, $h$ is necessarily positive-definite, nonsingular, symmetric, and has entries which vary smoothly in $x$. Let $(x_1,\ldots,x_d,\xi_1,\ldots,\xi_d)$ and $(z_1,\ldots,z_n, \zeta_1,\ldots,\zeta_n)$ be local canonical coordinates for $T^*H$ and $T^*M$, respectively. Recall, the Hamilton vector field associated with the symbol $p_H$ is given by
\[
	H_{p_H} = \sum_{k = 1}^d \left( \frac{\partial p_H}{\partial \xi_k} \frac{\partial}{\partial x_k} - \frac{\partial p_H}{\partial x_k} \frac{\partial}{\partial \xi_k} \right).
\]
We will be concerned with the coefficients $\partial p_H/\partial \xi_k$ of the spacial part of the vector. In particular, we will want to relate these coefficients to the corresponding coefficients of $H_{p_M}$. An elementary computation yields
\[
	\frac{\partial p_H}{\partial \xi_k}(x,\xi) = \frac{1}{p_H(x,\xi)} \sum_{j = 1}^d g_H^{jk}(x) \xi_j
	\quad \text{ and } \quad
	\frac{\partial p_M}{\partial \zeta_k}(z,\zeta) = \frac{1}{p_M(z,\zeta)} \sum_{j = 1}^n g_M^{jk}(z) \zeta_j.
\]
In particular we have
\begin{equation}\label{hamiltonian computation}
	\frac{\partial p_M}{\partial \zeta_k}(ix,\zeta) = \frac{1}{p_M(ix,\zeta)}
	\begin{cases}
		\displaystyle \sum_{j = 1}^d g_H^{jk}(x) \zeta_j & 1 \leq k \leq d, \\
		\displaystyle \sum_{j = d+1}^n h^{jk}(x) \zeta_j & d+1 \leq k \leq n
	\end{cases}
\end{equation}
where by yet another abuse of notation we take $h^{jk}(x)$ to be the entry of the inverse matrix of $h(x)$ in row $j-d$ and column $k-d$. From this we obtain the following convenient formula:
	\begin{equation}\label{hamiltonian identity}
		\frac{\partial p_M}{\partial \zeta_k}(ix,\zeta) = \frac{p_H(x,\pi\zeta)}{p_M(ix,\zeta)} \frac{\partial p_H}{\partial \xi_k}(x,\pi\zeta) \qquad \text{ for } k \in \{1,\ldots,d\}.
	\end{equation}
	We will need \eqref{hamiltonian identity} and the following elementary lemma for the proof of Lemma \ref{isolated at origin}.

\begin{lemma}\label{st bounds}
	There exists an open neighborhood $\mathcal O$ of the origin in $\R^2$ and a positive constant $C$ both only depending on $H$ and $M$ such that if $(s,t) \in \mathcal O$ and $(s,t,x,\zeta) \in F$ for some $x,\zeta$, then
	\[
		\left| \frac{s}{t} - 1 \right| \leq C |t|.
	\]
\end{lemma}

\begin{proof}[Proof of Lemma \ref{isolated at origin}]
	Suppose there exists a sequence of points $(s,t,x,\zeta) \in F$ for which $(s,t) \to 0$. Since $p_H/p_M$ is homogeneous of degree $0$ and since $S^*_H M$ is compact, we may select a subsequence for which $p_M(ix,\zeta) = 1$ and $(x,\zeta)$ converges. Let us examine the condition
	\[
		(i \otimes I) G_H^{s}(x,\pi\zeta) = (I \otimes \pi) G_M^t(ix,\zeta)
	\]
	in $F$. In local coordinates, we write
	\[
		\left(\frac{s}{t} \right) \frac{1}{s} \Big( (i \otimes I) G_H^{s}(x,\pi \zeta) - (ix,\pi \zeta) \Big)
		= \frac1t \Big((I \otimes \pi) G_M^t(ix,\zeta) - (ix,\pi \zeta)\Big).
	\]
	By taking a limit and invoking Lemma \ref{st bounds}, we find for all $\epsilon > 0$ a term in the sequence for which
	\[
		|(i \otimes I) H_{p_H}(x,\pi\zeta) - (I \otimes \pi) H_{p_M}(ix, \zeta) | < \epsilon.
	\]
	We look to the spacial components of the Hamilton vectors and find
	\[
		\sum_{k = 1}^d \left| \frac{\partial p_H}{\partial \xi_k}(x,\pi\zeta) - \frac{\partial p_M}{\partial \zeta_k}(ix,\zeta) \right|^2 + \sum_{k = d+1}^n \left| \frac{\partial p_M}{\partial \zeta_k}(ix,\zeta) \right|^2 < \epsilon^2.
	\]
	In light of \eqref{hamiltonian identity}, we have
	\[
		\left| 1 - \frac{p_M(ix,\zeta)}{p_H(x,\pi \zeta)} \right|^2 \sum_{k = 1}^d \left| \frac{\partial p_M}{\partial \zeta_k}(ix,\zeta) \right|^2 + \sum_{k = d+1}^n \left| \frac{\partial p_M}{\partial \zeta_k}(ix,\zeta) \right|^2 < \epsilon^2.
	\] 
	Since $p_M(ix,\zeta) = 1$, we can use this inequality to force $p_H(x,\pi\zeta)/p_M(ix,\zeta)$ as close to $1$ as we wish by taking $\epsilon$ small. This contradicts the condition that $p_H(x,\pi\zeta)/p_M(ix,\zeta) \in \supp \chi$ in $F$.
\end{proof}

\begin{proof}[Proof of Proposition \ref{clean at origin}]
	By \eqref{simple clean} and Lemma \ref{isolated at origin}, it suffices to show that
	\[
		T_p F_0 = \{ (s',t',x',\zeta') \in T_p F : (i \otimes I)(s'H_{p_H}) = (I\otimes \pi)(t'H_{p_M})\}
	\]
	for each $p = (0,0,x,\zeta) \in F_0$. This amounts to showing $s' = t' = 0$. The computation of the excess follows by $\dim F_0 = n + d$.
	
	By homogeneity, assume $p_M(ix,\zeta) = 1$. In our usual local coordinates,
	\begin{align*}
		s' \frac{\partial p_H}{\partial \xi_k}(x,\pi \zeta) &= t' \frac{\partial p_M}{\partial \zeta_k}(ix,\zeta) \qquad \text{ for } k \in \{1,\ldots,d\} \text{ and } \\
		0 &= t' \frac{\partial p_M}{\partial \zeta_k}(ix,\zeta) \qquad \text{ for } k \in \{d+1,\ldots,n\}.
	\end{align*}
	If $t' \neq 0$, then by \eqref{hamiltonian computation}, $\zeta_k = 0$ for each $k \in \{1,\ldots,d\}$. We would then have $p_M(ix,\zeta) = p_H(x,\pi\zeta)$, which is prohibited on $F$. Hence, $t' = 0$ and we have
	\[
		s' \frac{\partial p_H}{\partial \xi_k}(x,\pi \zeta) = 0 \qquad \text{ for } k \in \{1,\ldots,d\}.
	\]
	By \eqref{hamiltonian identity},
	\[
		\frac{p_H(x,\pi\zeta)}{p_M(ix,\zeta)} \frac{\partial p_H}{\partial \xi_k}(x,\pi \zeta) = \frac{\partial p_M}{\partial \zeta_k}(ix,\zeta) \qquad \text{ for each } k \in \{1,\ldots,k\}.
	\]
	Hence if $s' \neq 0$, each $\partial p_M/\partial \zeta_k$ must vanish for $k \in \{1,\ldots,d\}$, and so again by \eqref{hamiltonian computation}, $\zeta_k = 0$ for each $k \in \{1,\ldots,d\}$. It follows $p_H(x,\pi\zeta) = 0$, which is also prohibited on $F$. So $s' = 0$ and the proof is complete.
\end{proof}

\begin{proof}[Proof of Lemma \ref{st bounds}]
	For $|s|$ and $|t|$ smaller than the injectivity radii of $H$ and $M$, respectively,
	\[
		\frac{s}{t} = \frac{d_H(x,y)}{d_M(ix,iy)}
	\]
	where $y$ is the point on $H$ to which both $G_H^{s}(x,\pi\zeta)$ and $G_M^t(ix, \zeta)$ project, and where $d_H$ and $d_M$ are the Riemannian metrics on $H$ and $M$, respectively. Since $i : H \to M$ is an isometric embedding, $d_H(x,y) \geq d_M(ix,iy)$ always, so we need only show
	\[
		d_H(x,y) - d_M(ix,iy) \leq C d_M(ix,iy)^2
	\]
	wherever $d_H(x,y)$ is bounded by a small constant. In fact in normal coordinates about $ix$ in $M$, $d_M(ix,iy)=|iy|$, and $d_H(x,y)^2$ is a smooth function in $y$ satisfying $$d_H(x,y)^2=|y|^2+O(|y|^3).$$
	Then we have $$d_H(x,y)-d_M(ix,iy)=\dfrac{d_H(x,y)^2-d_M(ix,iy)^2}{d_H(x,y)+d_M(ix,iy)}=O(|y|^2),$$
	as desired. Furthermore, the constants implicit in the big-$O$ notation vary continuously with $x$. The lemma follows by compactness.
\end{proof}

\subsection{The Symbolic Data of the Composition}

We are nearly prepared to compute the principal symbol of $U \circ A \circ \delta_{i \times i}$ in a neighborhood of the origin.

\begin{proposition}\label{comp symbolic data}
	The restriction of $U \circ A \circ \delta_{i \times i}$ to the open neighborhood $\mathcal O$ of Lemma \ref{isolated at origin} is a half-density distribution in $I^{n - 3/2}(\R^2; \mathcal C_{\mathcal O} \circ \Lambda)$ where
	\[
		C \circ \Lambda_{\mathcal O} = \{(0,\sigma,0,\tau) \in T^*\R^2 : \sigma,\tau < 0 \text{ and } \sigma/\tau \in \supp \chi\}
	\]
	with principal symbol
	\begin{multline*}
		(2\pi)^{-n + \frac32} \vol(S^{d-1})\vol(S^{n-d-1}) \vol(H) \\
	 	\cdot \chi(\sigma/\tau)^2 (-\sigma)^{d-1} (-\tau)^{n-d-1} (1 - \sigma^2/\tau^2)^{\frac{n-d-2}{2}} |d\sigma \, d\tau|^{1/2},
	\end{multline*}
	wherever $\sigma/\tau$ belongs to a given interval $[a,b]\subset (0,1)$, modulo multiplication by a complex unit.
\end{proposition}

The order $n - 3/2$ is computed by
\[
	\ord(U \circ A \circ \delta_{i \times i}) = \ord (U \circ A) + \ord \delta_{i \times i} + e/2.
\]
Recall $\ord (U \circ A) = -1/2$ from Proposition \ref{U symbolic data}, $\ord \delta_{i \times i} = (n-d)/2$ from Proposition \ref{delta symbolic data}, and $e = n + d - 2$ from Proposition \ref{clean at origin}.

We now show how this proposition concludes the proof of Theorem \ref{main FIO}. Recall
\[
	U \circ A \circ \delta_{i \times i} = \widehat N_B |ds \, dt|^{1/2}.
\]
Let $\rho$ be a Schwartz-class function on $\R^2$ whose Fourier support is contained in $\mathcal O$. We test both sides against the oscillating half-density $\widehat \rho(s,t) e^{-i(s \sigma + t\tau)} |ds \, dt|^{1/2}$ to obtain
\[
	(U \circ A \circ \delta_{i \times i}, \widehat \rho(s,t) e^{-i(s \sigma + t\tau)} |ds \, dt|^{1/2}) = (2\pi)^2 \rho * N_B(-\sigma, -\tau).
\]
Referring to Proposition \ref{comp symbolic data} along with \cite[Proposition 25.1.5]{HIV} and the discussion preceding \cite[Theorem 25.1.9]{HIV}, we obtain
\begin{multline}\label{almost there convolution}
	\rho * N_B(-\sigma,-\tau) = (2\pi)^{-n} \vol(S^{d-1})\vol(S^{n-d-1}) \vol(H) \\
	 	\widehat \rho(0) (-\sigma)^{d-1} (-\tau)^{n-d-1} (1 - \sigma^2/\tau^2)^{\frac{n-d-2}{2}} + O(|(\sigma,\tau)|^{n-3})
\end{multline}
for $\sigma/\tau \in [a,b]$. Strictly speaking, the main term is only correct up to multiplication by a complex unit. But, this complex unit is $1$ anyway after taking $\rho \geq 0$ and recalling $N_B \geq 0$. Finally, the theorem follows after a change of variables $(\mu,\lambda) = (-\sigma,-\tau)$.

To compute the symbol, we first require a specialized basis for the tangent space of $T_H^*M$. In what follows, we identify our fiber product $F_0$ with
\[
	F_0 = \{(x,\zeta) \in T_H^*M \setminus 0 : \frac{p_H(x,\pi \zeta)}{p_M(ix,\zeta)} \in \supp \chi \}
\]
and, via the obvious inclusion, view $T_{(x_0,\zeta_0)} F_0$ as a subspace of $T_{(ix_0,\zeta_0)} T^*M$. For shorthand, we will also use
\[
	c = \frac{p_H(x_0,\pi \zeta_0)}{p_M(ix_0,\zeta_0)}.
\]

\begin{lemma}\label{coordinates lemma} Fix $(x_0,\zeta_0)$ in $F_0$. Take a system of local coordinates $(z_1,\ldots,z_n)$ of $M$ about $ix_0$ and consider the symplectic basis $e_1,\ldots,e_n, f_1,\ldots,f_n$ of $T_{(ix_0,\zeta_0)}T^*M$ given by
\[
	e_j = \frac{\partial}{\partial z_j} \qquad \text{ and } \qquad f_j = \frac{\partial}{\partial \zeta_j} \qquad j \in \{1,\ldots,n\}
\]
in canonical coordinates $(z_1,\ldots,z_n,\zeta_1,\ldots,\zeta_n)$. We may select these local coordinates such that the following hold.
\begin{enumerate}
	\item $e_1,\ldots,e_d,f_1,\ldots,f_n$ is a basis for $T_{(x_0,\zeta_0)} F_0$.
	\item $di e_j = e_j$ for all $j \leq d$.
	\item $d\pi f_j = f_j$ for all $j \leq d$, and $0$ if $j > d$.
	\item $dp_H e_j = 0$ for all $j$.
	\item $\displaystyle dp_H f_j = \begin{cases} 0 & j < d \\ 1 & j = d. \end{cases}$
	\item $dp_M e_j = 0$ for $j \leq d$.
	\item $\displaystyle dp_M f_j = \begin{cases}
		c & j = d, \\
		\sqrt{1 - c^2} & j = n, \\
		0 &\text{ otherwise}.
		\end{cases}$
	\item $H_{p_H} = e_d$.
	\item $H_{p_M} = c e_d + \sqrt{1 - c^2} e_n$ plus a linear combination of $f_{d+1}, \ldots, f_n$.
\end{enumerate}
\end{lemma}

There are some minor abuses of notation in the lemma above. First, (2) is tautological since we have identified $T_{(x_0,\zeta_0)} F_0$ as a subspace of $T_{(ix_0,\zeta_0)} T^*M$. Second, we are viewing $d\pi$ as an operator on $T_{(ix_0,\zeta_0)} T^*M$ rather than a map to $T_{(x_0,\pi\zeta_0)} T^*H$ and, strictly speaking, should be written $I \otimes d\pi$. Alternatively, we can make precise sense of the notation if we use the symplectic basis to pull everything back to the model Euclidean symplectic space $\R^{2n}$.

\begin{proof}
	Consider geodesic normal coordinates $(x_1,\ldots,x_d)$ about $x_0$ in $H$. Without loss of generality, we select these coordinates so that $(0,\ldots,0,s)$ parametrizes the geodesic traced by $G^s_H(x_0,\pi\zeta_0)$. It follows that
	\[
		(x_0,\pi \zeta_0) = (0,\ldots,0,p_H(x_0,\pi \zeta_0))
	\]
	in canonical local coordinates $(x,\xi)$, and 
	\begin{equation}\label{coordinates lemma 1}
		\frac{\partial p_H}{\partial \xi_j}(x_0,\pi \zeta_0) = \begin{cases}
			1 & j = d \\
			0 & j < d.
		\end{cases}
	\end{equation}
	Since $g_H^{-1}(x) = I + O(|x|^2)$ in normal coordinates, we also have
		\begin{equation} \label{coordinates lemma 2}
		\frac{\partial p_H}{\partial x_\ell}(x_0,\pi\zeta_0) = \frac{1}{2p_H} \sum_{j,k = 1}^d \frac{\partial g_H^{jk}}{\partial x_\ell}(x_0) \xi_j \xi_k = 0 \qquad \text{ for } \ell \in \{1,\ldots,d\}.
	\end{equation}
	
	We extend to coordinates in $M$ by first selecting a smooth orthonormal frame $v_{d+1}, \ldots, v_{n}$ of vectors perpendicular to $H$ and taking coordinates $(z_1,\ldots,z_n)$ defined by the map
\[
	(z_1,\ldots,z_n) \mapsto \exp(z_{d+1}v_{d+1}(z') + \cdots + z_n v_{n}(z'))
\]
where for shorthand $z' = (z_1,\ldots,z_d)$. We note
\begin{equation}\label{coordinates lemma 3.5}
	g_M^{-1}(z',0) = \begin{bmatrix}
		g_H^{-1}(z') & 0 \\
		0 & I
	\end{bmatrix}
	=
	\begin{bmatrix}
		I + O(|z'|^2) & 0 \\
		0 & I
	\end{bmatrix}
\end{equation}
Without loss of generality, we select $v_n$ such that the $\zeta$-gradient of $p_M$ is a linear combination of $\partial/\partial z_d$ and $v_n = \partial/\partial z_n$. And, since $g_M$ is the identity at the origin and $|\nabla_\zeta p_M| = 1$, it follows
\begin{equation}\label{coordinates lemma 5}
	\frac{\partial p_M}{\partial \zeta_j} = \begin{cases}
		p_H/p_M & j = d, \\
		\sqrt{1 - p_H^2/p_M^2} & j = n, \\
		0 &\text{otherwise}
	\end{cases}
\end{equation}
in canonical local coordinates.

Properties (1), (2), and (3) follow by construction. (4) follows from
\[
	dp_H e_j = \frac{\partial p_H}{\partial x_j}(x_0,\pi \zeta_0)
\]
and \eqref{coordinates lemma 2}. (5) follows from \eqref{coordinates lemma 1}. (6) follows from property (4) and \eqref{coordinates lemma 3.5} if $j \leq d$, and from
\[
	dp_M e_j = \frac{\partial p_M}{\partial z_j}(0) = \frac{1}{2p_M} \sum_{k,\ell = 1}^n \frac{\partial g^{k\ell}_M}{\partial z_j}(0) = \frac{1}{2p_M} \sum_{k,\ell = 1}^d \frac{\partial g^{k\ell}_H}{\partial x_j}(0).
\]
Finally, (7) follows from \eqref{coordinates lemma 5}. (8) and (9) follow from (4) through (7) and a computation of $H_{p_M}$ and $H_{p_H}$ in local coordinates.
\end{proof}

Now we prove Proposition \ref{comp symbolic data}.
In what follows, we assume familiarity with the composition calculus in \cite{DG75}. As in \cite{DG75}, given a vector space $V$ and a real number $\alpha$, we use $|V|^\alpha$ to denote the $1$-dimensional vector space of $\alpha$-densities on $V$. Please see also \cite{HIV} for an alternate exposition of the symbol calculus of FIOs, and to \cite{GS} for the half-density formalism.

\begin{proof}[Proof of Proposition \ref{comp symbolic data}]
Like before, fix $p = (x,\zeta) \in F_0$ and set
\begin{align*}
	a &= (0,-p_H(x,\pi \zeta),0,-p_M(ix,\zeta)) \qquad \text{ and } \\
	b &= (x, \pi \zeta, x, -\pi \zeta, ix, -\zeta, ix, \zeta).
\end{align*}
We let $\alpha : F_0 \to \mathcal C \circ \Lambda$ be the projection of the fiber onto the composition and let $d\alpha$ be its linearization $T_p F_0 \to T_a \mathcal C \circ \Lambda$. The procedure in \cite{DG75} identifies an object in $|\ker d\alpha| \otimes |T_a \mathcal C \circ \Lambda|^{1/2}$ which is then integrated over the `excess fiber' $\alpha^{-1}(a)$ to obtain the half-density symbol on $\mathcal C_{\mathcal O} \circ \Lambda$. 

Following the procedure in \cite{DG75}, we will identify the object in $|\ker d\alpha| \otimes |T_a\mathcal C \circ \Lambda|^{1/2}$ with three steps.
\begin{enumerate}
\item Let $\tau : T_{(a;b)} \mathcal C \times T_b \Lambda \to T_b T^*(H^2 \times M^2)$ given by $\tau((a';b'), c') = b' - c'$. We use the exact sequence
\begin{equation}\label{long exact sequence}
	0 \longrightarrow T_p F_0 \longrightarrow T_{(a;b)} \mathcal C \times T_{b} \Lambda \overset{\tau}{\longrightarrow} T_b T^*(H^2 \times M^2) \longrightarrow \coker \tau \longrightarrow 0
\end{equation}
along with the symbols on $\mathcal C$ and $\Lambda$ and the symplectic half-density on $T^*(H^2 \times M^2)$ to obtain a linear isomorphism $|T_pF_0|^{-1/2} \simeq |\coker \tau|^{-1/2}$, which we make explicit in Lemma \ref{coker tau lemma}, below.

\item The spaces $\coker \tau$ and $\ker \alpha$ are nondegenerately paired by the symplectic form on $T^*(H^2 \times M^2)$ (`canonically paired' in Duistermaat and Guillemin's language). This induces a linear isomorphism $|\coker \tau|^{-1/2} \simeq |\ker \alpha|^{1/2}$. Both the pairing and the isomorphism are made explicit in Lemma \ref{canonically paired lemma}, below.

\item The short exact sequence
\begin{equation}\label{short exact sequence}
	0 \longrightarrow \ker d\alpha \longrightarrow T_p F_0 \overset{d\alpha}{\longrightarrow} T_a(\mathcal C \circ \Lambda) \longrightarrow 0 
\end{equation}
yields a unique element in $|\ker d\alpha|^{1/2} \otimes |T_p F_0|^{-1/2} \otimes  |T_a \mathcal C \circ \Lambda|^{1/2}$, which by the isomorphisms $|T_pF_0|^{-1/2} \simeq |\coker \tau|^{-1/2} \simeq |\ker d\alpha|^{1/2}$ from (1) and (2), identifies a unique element of $|\ker d\alpha| \otimes |T_a \mathcal C \circ \Lambda|^{1/2}$. This is made explicit in Lemma \ref{short exact lemma}, below.
\end{enumerate}

\begin{lemma}\label{coker tau lemma}
	The exact sequence \eqref{long exact sequence} induces a linear isomorphism $|T_p F_0|^{-1/2} \simeq |\coker \tau|^{-1/2}$ which, given the $-\frac12$-density on $T_p F_0$ which assigns the value $1$ to the basis for $T_p F_0$ of Lemma \ref{coordinates lemma}, gives a $-\frac12$-density on $\coker \tau$ which assigns 
	\[
		(2\pi)^{-\codim H/2 + 1/2} \chi(c)^2 (1 - c^2)^{-1/4}
	\]
	to the basis of $\coker \tau$ obtained as the image of
	\begin{align}\label{coker tau basis}
		&(0,0,f_j,0) && j \leq d \\
		\nonumber &(0,e_j,0,0) && j < d \\
		\nonumber &(0,0,0,e_j) && d < j < n
	\end{align}
	under the quotient map $T_b T^*(H^2 \times M^2) \to \coker \tau$. 
\end{lemma}

The pairing of $\ker d\alpha$ and $\coker \tau$ hinted in (2) is realized as follows. By the arguments in \cite{DG75}, the image of $\ker d\alpha$ in $T_b T^*(H^2 \times M^2)$ via $T_p F_0 \to T_b T^*(H^2 \times M^2)$ is precisely the symplectic complement of $\im \tau$. The symplectic form then induces a well-defined nondegenerate bilinear mapping
\[
	\ker d\alpha \times \coker \tau \to \R.
\]

\begin{lemma} \label{canonically paired lemma}
	The pairing of $\ker d\alpha$ and $\coker \tau$ induces a linear isomorphism $|\coker \tau|^{1/2} \simeq |\ker d\alpha|^{-1/2}$ that, given a half-density on $\coker \tau$ which assigns the value $1$ to the basis in Lemma \ref{coker tau lemma}, gives a $-\frac12$-density on $\ker d\alpha$ which assigns the value $1$ to the basis
	\[
		e_1, \ldots, e_d, f_1, \ldots, f_{d-1}, f_{d+1},\ldots,f_{n-1}.
	\]
\end{lemma}

\begin{lemma} \label{short exact lemma} The short exact sequence \eqref{short exact sequence} and the isomorphisms of the previous two lemmas identify a unique element in $|\ker d\alpha| \otimes |T_a \mathcal C \circ \Lambda|^{1/2}$ which assigns 
the value 
\[
	(2\pi)^{-\codim H/2+1/2} (1 - c^2)^{-1/2}
\]
to the pair of bases
\[
	e_1,\ldots,e_d,f_1,\ldots,f_{d-1},f_{d+1},\ldots, f_{n-1} \qquad \text{ and } \qquad (h,0),(0,h)
\]
for $\ker d\alpha$ and $T_a \mathcal C \circ \Lambda$, respectively.
\end{lemma}

Finally, we integrate the object identified in Lemma \ref{short exact lemma} over the excess fiber $\alpha^{-1}(0,\sigma,0,\tau)$. We fix coordinates as we have above and integrate in the $\zeta$ variables first and $x$ variables second. We note that $(x_0,\zeta) \in \alpha^{-1}(0,\sigma,0,\tau)$ if and only if
\[
	|(\zeta_1,\ldots,\zeta_d)| = - \sigma \qquad \text{ and } \qquad |(\zeta_1,\ldots,\zeta_n)| = -\tau,
\]
which may be rewritten as the cartesian product of the sphere of radius $-\sigma$ in $\R^d$ and the sphere of radius $\sqrt{\tau^2 - \sigma^2}$ in $\R^{n-d}$. Integrating yields an object in $|T_{x_0} H| \otimes |T_a \mathcal C \circ \Lambda|^{1/2}$ which assigns the value
\[
	(2\pi)^{-\codim H/2+1/2} \vol(S^{d-1})\vol(S^{n-d-1}) (-\sigma)^{d-1} (-\tau)^{n-d-1} (1 - \sigma^2/\tau^2)^{\frac{n-d-2}{2}} 
\]
to the pair of bases $e_1,\ldots, e_d$ and $(h,0),(0,h)$. Since $e_1,\ldots,e_d$ are orthonormal at $x_0$ with respect to the Riemannian inner product, integration over $H$ yields the half-density
\begin{multline*}
	(2\pi)^{-\codim H/2+1/2} \vol(S^{d-1})\vol(S^{n-d-1}) \vol(H) \\
	 (-\sigma)^{d-1} (-\tau)^{n-d-1} (1 - \sigma^2/\tau^2)^{\frac{n-d-2}{2}} |d\sigma \, d\tau|^{1/2}.
\end{multline*}
Proposition \ref{comp symbolic data} follows by [DG, Theorem 5.4] or [HorIV, 25.2.3], which multiplies in an additional complex unit times $(2\pi)^{-e/2} = (2\pi)^{-(n+d-2)/2}$.
\end{proof}

All that remains is to prove Lemmas \ref{coker tau lemma}, \ref{canonically paired lemma}, and \ref{short exact lemma}.

\begin{proof}[Proof of Lemma \ref{coker tau lemma}]
The exact sequence \eqref{long exact sequence} gives us an identification
\[
	|\coker \tau|^{-1/2} \simeq |T_p F_0|^{-1/2} \otimes |T_{(a;b)} \mathcal C \times T_b \Lambda|^{1/2} \otimes |T_b T^*(H^2 \times M^2)|^{-1/2}.
\]
We will begin by fixing an element in $|T_p F_0|^{-1/2}$ which takes the basis (1) of Lemma \ref{coordinates lemma} to $1$ and, through a process of forward maps, basis extensions, and change of basis operations, we will obtain the desired valuation of an element in $|\coker \tau|^{-1/2}$ on the basis in the lemma.

{In what follows, we let $g,h$ constitute the standard symplectic basis of any tangent space to $T^*\R$.}
Recalling Propositions \ref{U symbolic data} and \ref{delta symbolic data} and appealing to Lemma \ref{coordinates lemma}, $T_{(a;b)} \mathcal C$ has a basis
\begin{align*}
		&(g,0;0,-e_d, 0,0)\\
		&(0,g;0,0,0,ce_d + \sqrt{1 - c^2}e_n + *) \\
		&(0,0;e_j,e_j,0,0) && j \leq d\\
		&(0,0;f_j,-f_j,0,0) && j < d\\
		&(-h,0;f_d,-f_d,0,0) \\
		&(0,0;0,0,e_j,e_j) && j \leq n \\
		&(0,0;0,0,-f_j,f_j) && j \neq d,n\\
		&(0,-ch;0,0,-f_d,f_d) \\
		&(0,-\sqrt{1 - c^2}h; 0,0,-f_n,f_n)
\end{align*}
where the $*$ is a linear combination of $f_{d+1},\ldots,f_n$, and $T_{b} \Lambda$ has a basis
\begin{align*}
		&(e_j,0,e_j,0) && j \leq d \\
		&(0,e_j,0,e_j) && j \leq d \\
		&(f_j,0,-f_j,0) && j \leq d \\
		&(0,f_j,0,-f_j) && j \leq d \\
		&(0,0,-f_j,0) && j > d \\
		&(0,0,0,-f_j) && j > d. \\
\end{align*}
Recall the maps $T_pF_0 \to T_{(a;b)}\mathcal C$ and $T_p F_0 \to T_b\Lambda$ are given by
\begin{align*}
	(x',\zeta') &\mapsto (0,-dp_H(x',\pi\zeta'),0,-dp_M(ix',\zeta); x', \pi\zeta', x', -\pi \zeta', ix', -\zeta', ix', \zeta')\\
	(x',\zeta') &\mapsto (x', \pi\zeta', x', -\pi \zeta', ix', -\zeta', ix', \zeta')
\end{align*}
Appealing to Lemma \ref{coordinates lemma}, the basis $e_1,\ldots,e_d,f_1,\ldots,f_n$ of $T_p F_0$ pushes forward through the map $T_b F_0 \to T_{(a;b)} C \times T_b \Lambda$ to
\begin{align*}
	&(0, 0; e_j, e_j, e_j, e_j) \times (e_j, e_j, e_j, e_j) && j \leq d \\
	&(0, 0; f_j, -f_j, -f_j, f_i) \times (f_j, -f_j, -f_j, f_i) && j < d \\
	&(-h, -ch; f_d, -f_d, -f_d, f_d) \times (f_d, -f_d, -f_d, f_d) \\
	&(0, 0; 0, 0, -f_j, f_j) \times (0, 0, -f_j, f_j) && d < j < n \\
	&(0, -\sqrt{1 - c^2}h; 0, 0, -f_n, f_n) \times (0, 0, -f_n, f_n)
\end{align*}
Our current goal is to complete this to a basis for the product $T_{(a;b)} \mathcal C \times T_b \Lambda$. The evaluation of the symbolic half-density on the product will determine a unique element of $|(T_{(a;b)} \mathcal C \times T_b \Lambda)/\ker \tau|^{1/2}$. We add in $d + n$ elements of the form $0 \times v$ where $v$ is a basis element of $T_b \Lambda$. We also add in $2 + 2d + 2n$ elements of the form $0 \times v$ where $v$ are the basis elements of $T_{(a;b)}C$. Specifically, we extend to a basis of the product by adding in elements
\begin{align}\label{basis extension 1}
	&0 \times (e_j, 0, e_j, 0) && j \leq d\\
	\nonumber &0 \times (0, e_j, 0, e_j) && j \leq d\\
	\nonumber &0 \times (f_j, 0, -f_j, 0) && j \leq d \\
	\nonumber &0 \times (0, -f_j, 0, f_j) && j \leq d \\
	\nonumber &0 \times (0, 0, -f_j, 0) && d < j \leq n \\
	\nonumber &0 \times (0, 0, 0, f_j) && d < j \leq n\\
	\nonumber &(g,0;0,-e_d, 0,0) \times 0\\
	\nonumber &(0,g;0,0,0,ce_d + \sqrt{1 - c^2}e_n + *) \times 0 \\
	\nonumber &(0, 0; 0, 0, e_j, e_j) \times 0 && j \leq n \\
	\nonumber &(0,0;0,0,-f_j,f_j) \times 0 && j < d\\
	\nonumber &(0,-ch;0,0,-f_d,f_d) \times 0
\end{align}
A linear transformation consisting entirely of determinant $\pm 1$ row operations and perhaps some permutations takes the extended basis to the product of the bases of $T_{(a;b)} C$ and $T_b \Lambda$. Appealing to the symbolic data in Lemmas \ref{U symbolic data} and \ref{delta symbolic data}, the half-density on the product $T_{(a;b)} C \times T_b \Lambda$ assigns the value
	\[
		(2\pi)^{-\codim H/2 + 1/2} \chi(c)^2
	\]
	to the product basis. Hence, the desired element of $|(T_{(a;b)} C \times T_b \Lambda)/\ker \tau|^{1/2}$ assigns the same value to the image of the list \eqref{basis extension 1} by the quotient map $T_{(a;b)} C \times T_b \Lambda \to (T_{(a;b)} C \times T_b \Lambda)/\ker \tau$.
	
	Next, we map \eqref{basis extension 1} forward to $T_bT^*(H^2 \times M^2)$ via $\tau$. In particular, we obtain
\begin{align*}
	&(-e_j, 0, -e_j, 0) && j \leq d \\
	&(0, -e_j, 0, -e_j) && j \leq d \\
	&(-f_j, 0, f_j, 0) && j \leq d \\
	&(0, f_j, 0, -f_j) && j \leq d \\
	&(0,0, f_j, 0) && d < j \leq n \\
	&(0,0,0,-f_j) && d < j \leq n \\
	&(0, -e_d, 0, 0) \\
	&(0,0,0, ce_d + \sqrt{1 - c^2} e_n + *) \\
	&(0,0, e_j, e_j) && j \leq n \\
	&(0,0, -f_j, f_j) && j < d \\
	&(0,0, -f_d, f_d).
\end{align*}
Recall the $*$ is a linear combination of $f_j$. We extend this by the $n + d - 2$ elements in \eqref{coker tau basis} and, after a sequence of determinant $\pm 1$ operations and a permutation of the basis elements, we obtain the basis
\begin{align*}
	&(e_j, 0, 0, 0) && j \leq d \\
	&(0,e_j,0,0) && j \leq d \\
	&(0,0, e_j, 0) && j \leq n \\
	&(0,0,0,e_j) && j < n \\
	&(0,0,0, \sqrt{1 - c^2} e_n) \\
	&(f_j, 0, 0, 0) && j \leq d \\
	&(0, f_j, 0, 0) && j \leq d \\
	&(0,0,f_j,0) && j \leq n \\
	&(0,0,0, f_j) && j \leq n,
\end{align*}
to which the symplectic $-\frac12$-density on $T_bT^*(H^2 \times M^2)$ assigns $(1 - c^2)^{-1/4}$. Finally, the desired element in $|\coker \tau|^{-1/2}$ assigns the product of these valuations,
\[
	(2\pi)^{-\codim H/2+1/2} \chi(c)^2 (1 - c^2)^{-1/4},
\]
to the image of the basis \eqref{coker tau basis} via the quotient $T_bT^*(H^2 \times M^2) \to \coker \tau$.
\end{proof}

\begin{proof}[Proof of Lemma \ref{canonically paired lemma}]
	The map $T_pF_0 \to T_bT^*(H^2 \times M^2)$ is
	\[
		(x',\zeta') \mapsto
		(x',d\pi \zeta',x',-d\pi \zeta',dix',-\zeta',dix',\zeta').
	\]
	Hence, the basis for $\ker d\alpha$ in the lemma maps forward to
	\begin{align}\label{canonically paired lemma basis}
		&(e_j,e_j,e_j,e_j) && j \leq d \\
		\nonumber &(f_j,-f_j,-f_j,f_j) && j < d \\
		\nonumber &(0,0,-f_j,f_j) && d < j < n.
	\end{align}
	The restriction of the symplectic form on $T_b T^*(H^2 \times M^2)$ to $\ker d\alpha \times \coker \tau$ has matrix
	\[
		\begin{bmatrix}
			I & 0 & 0 \\
			0 & I & 0 \\
			0 & 0 & -I
		\end{bmatrix}
	\]
	with respect to the products of the bases \eqref{canonically paired lemma basis} and \eqref{coker tau basis}. Since the absolute value of the determinant of this matrix is conveniently $1$, and since our $-\frac12$-density on $\coker \tau$ assigns $1$ to the basis in Lemma \ref{coker tau lemma}, the natural element in $|\ker d\alpha|^{1/2}$ assigns $1$ to the basis in the statement of the lemma.
\end{proof}

\begin{proof}[Proof of Lemma \ref{short exact lemma}]
The short exact sequence \eqref{short exact sequence} and the results of Lemmas \ref{coker tau lemma} and \ref{canonically paired lemma} induce a natural trivialization
\[
	1 \simeq |\ker d\alpha|^{1/2} \otimes |T_p F_0|^{-1/2} \otimes  |T_a \mathcal C \circ \Lambda|^{1/2} \simeq |\ker d\alpha| \otimes |T_a \mathcal C \circ \Lambda|^{1/2}.
\]
Take the element in $|T_p F_0|^{-1/2}$ which assigns the value $1$ to the basis in (1) of Lemma \ref{coordinates lemma}. Invoking both Lemma \ref{coker tau lemma} and Lemma \ref{canonically paired lemma}, the induced element in $|\ker d\alpha|^{1/2}$ assigns the value
\[
	(2\pi)^{-\codim H/2+1/2} \chi(c)^2 (1 - c^2)^{-1/4}
\]
to the basis in Lemma \ref{canonically paired lemma}. Note, the pushforward of $f_d,f_n$ through $d\alpha$ is
\[
	(-h,0),(-ch, -\sqrt{1 -c^2} h),
\]
and hence the object in $|\ker d\alpha| \otimes |T_a \mathcal C \circ \Lambda|^{1/2}$ assigns this same value to the pair of bases
\[
	e_1,\ldots,e_d,f_1,\ldots,f_{d-1},f_{d+1},\ldots, f_{n-1} \qquad \text{ and } \qquad (-h,0),(-ch, -\sqrt{1 -c^2} h).
\] 
Note, $(-h,0),(-ch,-\sqrt{1 - c^2} h)$ is the image of the basis $(h,0),(0,h)$ under a linear transformation of determinant $\sqrt{1 - c^2}$. Hence, our object in $|\ker d\alpha| \otimes |T_a \mathcal C \circ \Lambda|^{1/2}$ assigns the value
\[
	(2\pi)^{-\codim H/2+1/2} \chi(c)^2 (1 - c^2)^{-1/2}
\]
to the pair of bases
\[
	e_1,\ldots,e_d,f_1,\ldots,f_{d-1},f_{d+1},\ldots, f_{n-1} \qquad \text{ and } \qquad (h,0),(0,h).
\]
\end{proof}


\section{Proof of the Theorem \ref{main tauberian}} \label{TAUBERIAN}

Finally, we prove our Tauberian theorem. The idea at the core of the argument is based on that of Colin de Verdiere's version in \cite{CdV79} which obtains asymptotics of certain weighted counts of spectral measures in $\R^d$ over a homothetic family of regions with piecewise $C^1$ boundary. The main contribution here is to relate the remainder to the size of (a unit thickening of) the boundary of the region, rather than to a scaling parameter. This allows us to apply the Tauberian theorem to obtain estimates for sums over a joint spectrum in more exotic families of regions, as in Theorem \ref{main ladder}.

We introduce some notation. For a subset $\Omega \subset \R^n$ and $0 \leq a < b$, we define
\[
	\partial \Omega^{[a,b]} = \{x \in \overline{\Omega^c} : a \leq d(x,\Omega) \leq b \},
\]
and
\[
	\partial \Omega^{[-b,-a]} = \{x \in \overline{\Omega} : a \leq d(x,\Omega^c) \leq b \},
\]
and finally if $a < 0 < b$, we set
\[
	\partial \Omega^{[a,b]} =  \partial \Omega^{[a,0]} \cup \partial\Omega^{[0,b]} 
\]
As expected, $\partial \Omega^{[-1,1]}$ is the unit thickening of the boundary of $\Omega$.

To prove Theorem \ref{main tauberian}, we first record a couple helpful lemmas. The first is an observation about integration of order functions over thickenings of regions in $\R^n$. The second uses the first and allows us to control $N(\partial \Omega^{[0,r]})$ in terms of the integral of $m$ over $\partial \Omega^{[-1,1]}$.

\begin{lemma}\label{tauberian lemma 1}
	Let $m$ be an order function on $\R^n$, $\Omega$ a subset of $\R^n$, and $a,b$ real numbers for which $a < b$. Then, there exist a constant $C$ and an exponent $\nu$ depending only on $n$ and $m$ for which
	\[
		\int_{\partial\Omega^{[a,b]}} m(x) \, dx \leq C (1 + \max(|a|,|b|))^\nu \int_{\partial \Omega^{[-1,1]}} m(x) \, dx.
	\]
\end{lemma}

\begin{lemma}\label{tauberian lemma 2}
	We assume all hypotheses of Theorem \ref{main tauberian} except we do not require that $N(\Omega)$ be finite. Then for each $r \geq 1$,
	\[
		N(\partial \Omega^{[0,r]}) \leq C (1 + r)^\nu \int_{\partial \Omega^{[-1,1]}} m(x) \, dx
	\]
	for some constant $C$ and exponent $\nu$ which are independent of $r$ and $\Omega$.
\end{lemma}

We first prove Theorem \ref{main tauberian} and then the lemmas. Throughout the proofs, we use $a \lesssim b$ to mean that there exists a positive constant $C$ not depending on $\Omega$ for which
\[
	a \leq Cb.
\]
We also use $a \gtrsim b$ to denote $b \lesssim a$ and we use $a \approx b$ to mean $a \lesssim b$ and $a \gtrsim b$.

\begin{proof}[Proof of Theorem \ref{main tauberian}]
	We write
	\begin{align*}
		N(\Omega) - \rho * N(\Omega) &= \int_{\R^n} \chi_\Omega(y) \, dN(y) - 		\int_{\R^n} \int_{\R^n} \chi_\Omega(x) \rho(x - y) dN(y) \, dx \\
		&= \int_{\R^n} \int_{\R^n} (\chi_\Omega(y) - \chi_\Omega(x)) \rho(x - y) dN(y) \, dx.
	\end{align*}
	We note
	\[
		\chi_\Omega(x) - \chi_\Omega(y) = \begin{cases}
			1 & x \in \Omega \text{ and } y \in \Omega^c, \\
			-1 & x \in \Omega^c \text{ and } y \in \Omega,\\
			0 & \text{otherwise,}
		\end{cases}
	\]
	and hence
	\[
		|N(\Omega) - \rho * N(\Omega)| \leq \int_{\Omega} \int_{\Omega^c} \rho(x - y) \, dN(y) \, dx + \int_{\Omega^c} \int_{\Omega} \rho(x - y) \, dN(y) \, dx.
	\]
	We claim the first integral is bounded as
	\begin{equation}\label{tauberian claim}
		\int_{\Omega} \int_{\Omega^c} \rho(x - y) \, dN(y) \, dx \lesssim \int_{\partial \Omega^{[-1,1]}} m(x) \, dx.
	\end{equation}
	As we proceed, we will see that the claim holds similarly if we interchange $\Omega$ and $\Omega^c$, and hence the second integral satisfies the same bound. The theorem will follow.
	
	For $y \in \R^n$ and $r \geq 0$, we let
	\[
		A(y,r) = \{x \in \R^n : r \leq |x - y| \leq r+1\}
	\]
	denote the annulus centered at $y$ with inner radius $r$ and outer radius $r+1$. Since $\rho$ is Schwartz-class, we write
	\[
		\rho(x-y) \lesssim (1 + |x-y|)^{-K}
	\]
	for $K$ as large as we desire and $C$ which depends on $N$. Hence,
	\begin{align*}
		\int_{\Omega} \int_{\Omega^c} \rho(x - y) \, dN(y) \, dx &\leq \sum_{k = 0}^\infty  \int_{\partial \Omega^{[0,k+1]}} \int_{\Omega \cap A(y,k)} \rho(x - y) \, dx \, dN(y) \\
		&\lesssim \sum_{k = 0}^\infty (1 + k)^{-K} \int_{\partial \Omega^{[0,k+1]}} |\Omega \cap A(y,k)| \, dN(y) \\
		&\lesssim \sum_{k = 0}^\infty (1 + k)^{-K+n-1} N(\partial\Omega^{[0,k+1]})\\
		&\lesssim \sum_{k = 0}^\infty (1 + k)^{-K + n + \nu - 1} \int_{\partial \Omega^{[-1,1]}} m(x) \, dx
	\end{align*}
	where the last line follows from the second lemma. \eqref{tauberian claim} follows after taking $K > n + \nu$ and summing the convergent series.
	\end{proof}
	
	\begin{proof}[Proof of Lemma \ref{tauberian lemma 1}]	
	Let $P$ denote a maximal $1$-separated subset of $\partial \Omega^{[a,b]}$. Note
	\[
		m(x) \approx m(y) \qquad \text{ for } |x - y| \leq 1
	\]
	and, since the balls $B(x,1)$ of radius $1$ centered at points in $P$ cover $\partial \Omega^{[a,b]}$, we have
	\[
		\int_{\partial \Omega^{[a,b]}} m(x) \, dx \leq \sum_{x \in P} \int_{B(x,1)} m(y) \, dy \lesssim \sum_{x \in P} m(x).
	\]
	Now let $Q$ be a maximal $1$-separated subset of $\partial \Omega^{[-1/2,1/2]}$. Note that the balls $B(x,1/2)$ for $X \in Q$ are disjoint and lie entirely in $\partial \Omega^{[-1,1]}$. Hence,
	\[
		\int_{\partial \Omega^{[-1,1]}} m(x) \, dx \geq \sum_{x \in Q} \int_{B(x,1/2)} m(y) \, dy \gtrsim \sum_{x \in Q} m(x).
	\]
	Hence, it suffices to show
	\[
		\sum_{x \in P} m(x) \lesssim (1 + \max(|a|,|b|))^\nu \sum_{y \in Q} m(y).
	\]
	For shorthand let $r = 1 + \max(|a|,|b|)$. For each $x \in P$, let $y(x)$ denote some choice of $y \in Q$ for which $|x - y(x)| \leq r$. Such $y(x)$ must always exist. If $B(x,r) \cap Q$ were empty, then we would be able to place a point in $B(x,r-1) \cap \partial \Omega^{[-1/2,1/2]}$ which then would be $1$-separated from all other points in $Q$.
	
	Suppose $\nu'$ is the exponent in the bound on the order function $m$. That is,
	\[
		m(x) \lesssim (1 + |x-y|)^{\nu'}m(y).
	\]
	We have
	\begin{align*}
		\sum_{x \in P} m(x) &\lesssim r^{\nu'} \sum_{x \in P} m(y(x)) \\
		&= r^{\nu'} \sum_{y \in Q} \# \{ x : y(x) = y\} m(y) \\
		&\lesssim r^{\nu'+n} \sum_{y \in Q} m(y),
	\end{align*}
	where the last line follows from
	\[
		\# \{ x : y(x) = y\} \leq \# P \cap B(y,r) \lesssim r^n
	\]
	for each $y \in Q$. This completes the proof of the lemma.
\end{proof}
	
	\begin{proof}[Proof of Lemma \ref{tauberian lemma 2}]
	Since $\rho(0) > 0$, there exists $\delta \in (0,1]$ and a positive constant $c$ for which
	\[
		\rho(x) \geq c \qquad \text{ for } |x| \leq \delta.
	\]
	Hence,
	\begin{equation}\label{tauberian lemma 2 eq 1}
		\int_{B(x,\delta)} dN(y) \leq \frac{1}{c} \rho * N(x) \leq \frac1c m(x) \qquad \text{ for all } x \in \R^n.
	\end{equation}
	We let $\tilde N$ denote the restriction of $N$ to $\partial \Omega^{[0,r]}$ and have by Fubini's theorem,
	\begin{align*}
		N(\partial \Omega^{[0,r]}) &= \int_{\R^n} d \tilde N(y) \\
		&= \int_{\R^n}  \int_{B(y,\delta)} \frac{1}{|B(y,\delta)|} \, dx \, d\tilde N(y)\\
		&= \int_{\R^n} \int_{B(x,\delta)} \frac{1}{|B(x,\delta)|} d\tilde N(y) \, dx.
	\end{align*}
	Note, the inner integral is supported for $x$ in a $\delta$-thickening of the support of $\tilde N$, e.g. $\partial \Omega^{[-1,r+1]}$. This and $\tilde N \leq N$ yield
	\begin{align*}
		N(\partial \Omega^{[0,r]}) &\leq \int_{\partial \Omega^{[-1,r+1]}} \frac{1}{|B(x,\delta)|} \int_{B(x,\delta)} dN(y) \, dx\\
		&\leq \frac{1}{c|B(0,\delta)|} \int_{\partial \Omega^{[-1,r+1]}} m(x) \, dx,
	\end{align*}
	where the second line follows from \eqref{tauberian lemma 2 eq 1}. The proof is completed by Lemma \ref{tauberian lemma 1}.
\end{proof}


\end{document}